\documentclass[pdflatex,sn-mathphys-num]{sn-jnl}

\usepackage{graphicx}%
\usepackage{multirow}%
\usepackage{amsmath,amssymb,amsfonts}%
\usepackage{amsthm}%
\usepackage{enumerate}
\usepackage{mathrsfs}%
\usepackage[title]{appendix}%
\usepackage{xcolor}%
\usepackage{textcomp}%
\usepackage{manyfoot}%
\usepackage{booktabs}%
\usepackage{algorithm}%
\usepackage{algorithmicx}%
\usepackage{algpseudocode}%
\usepackage{listings}%
\usepackage{epstopdf}
\usepackage{mathtools}
\usepackage{subcaption} 

\theoremstyle{thmstyleone}%
\newtheorem{theorem}{Theorem}[section]

\newtheorem{lemma}{Lemma}[section]

\theoremstyle{thmstyletwo}%
\newtheorem{remark}{Remark}%

\theoremstyle{thmstylethree}%
\newtheorem{assumption}{Assumption}%

\newcommand{\R}{{\mathbb{R}}}

\newcommand{\BEQ}{\begin{equation}}
\newcommand{\EEQ}{\end{equation}}
\newcommand{\prox}{\mathbf{prox}}
\newcommand{\minimize}{\operatorname{minimize}}
\newcommand{\st}{\operatorname{subject~to}}
\newcommand{\mb}{\mathbf}

\begin{document}

\title[Primal-Dual Bundle Methods for Linear Equality-Constrained Problems]{Primal-Dual Bundle Methods for Linear Equality-Constrained Problems}


\author[1]{\fnm{Zhuoqing} \sur{Zheng}}\email{12433033@mail.sustech.edu.cn}
\author[1]{\fnm{Tao} \sur{Liu}}\email{liut6@sustech.edu.cn}

\author[1]{\fnm{Xuyang} \sur{Wu}}\email{wuxy6@sustech.edu.cn}

\affil[1]{\orgdiv{School of Automation and Intelligent Manufacturing}, \orgname{Southern University of Science and Technology}, \orgaddress{\street{1088 Xueyuan Avenue}, \city{Shenzhen}, \postcode{518055}, \state{Guangdong}, \country{China}}}


\abstract{Dual ascent (DA) and the method of multipliers (MM) are fundamental methods for solving linear equality-constrained convex optimization problems, and their dual updates can be viewed as the minimization of a proximal linear surrogate function of the negative Lagrange dual and augmented Lagrange dual function, respectively. However, the proximal linear surrogate function may suffer from low approximation accuracy, which leads to slow convergence of DA and MM. To accelerate their convergence, we adapt the proximal bundle surrogate framework that can incorporate a list of more accurate surrogate functions, to both the primal and the dual updates of DA and MM, leading to a family of novel primal-dual bundle methods. Our methods generalize the primal-dual gradient method, DA, the linearized MM, and MM. Under standard assumptions that allow for a broad range of surrogate functions, we prove theoretical convergence guarantees for the proposed methods. Numerical experiments demonstrate that our methods converge not only faster, but also significantly more robust with respect to the parameters compared to the primal-dual gradient method, DA, the linearized MM, and MM.}




\maketitle

\section{Introduction}\label{sec1}
This article addresses linear equality-constrained composite optimization problems:
\begin{equation}\label{problem}
\begin{split}
    \underset{x\in\R^n}{\minimize}~~~ & F(x) \overset{\triangle}{=} f(x)  + h(x)\\
    \st  ~~&Ax=b 
\end{split}
\end{equation}
where $A \in \mathbb{R}^{m \times n}$, $b \in \mathbb{R}^m$, $f:\mathbb{R}^n\rightarrow\mathbb{R}$ is convex and differentiable, and $h:\mathbb{R}^n\rightarrow\mathbb{R}\cup\{+\infty\}$ is convex but possibly non-differentiable. This 
problem has found numerous applications in machine learning \cite{wright2022optimization,sergio2015reinforement}, smart grid \cite{yang2016distributed}, and distributed optimization \cite{wu2023unifyMM}. 


A large number of algorithms have been proposed to solve problem \eqref{problem}, including the penalty-function-based methods \cite{mayne1979first,zangwill1967non,conn1973constrained,lan2013iteration}, projection-based methods \cite{levitin1966constrained,goldstein1964convex,bertsekas1976goldstein,bertsekas1982projected}, and the primal-dual methods \cite{polyak1970iterative,boyd2011ADMM,bertsekas2014constrained,He2012,yashtini2022convergence,malitsky2018first,chambolle2011first,luo2008convergence,Rocka1976proximalpoint,chen1994proximal,xu2017accelerated,douglas1956numerical,eckstein1992douglas,ouyang2015accelerated,attouch2008augmented,xu2021iteration,liu2019nonergodic,alghunaim2020linear,xu2021first,rockafellar1976augmented,liao2025bundlebasedMM}. Among these, primal-dual methods have become the most popular approaches, which solve problem \eqref{problem} by solving its Lagrange dual problem.
Two fundamental primal-dual methods are dual ascent (DA) and the method of multipliers (MM) \cite{boyd2011ADMM,Rocka1976proximalpoint,xu2021iteration,liu2019nonergodic}, which apply steepest descent and the proximal point method to minimizing the negative dual function of problem \eqref{problem}, respectively. Notably, DA requires strong convexity of $f$ and MM eliminates this requirement. To reduce the computational complexity per iteration or to accelerate the convergence, several variants of DA and MM were proposed. For example, the primal-dual gradient method \cite{alghunaim2020linear} and the linearized method of multipliers (linearized MM) \cite{xu2017accelerated,xu2021first} replaces $f$ in the primal update of DA and MM with a simple proximal-linear approximation, respectively; The proximal method of multipliers (PMM) \cite{chen1994proximal,rockafellar1976augmented} adds a weighted proximal term on the primal update, making the primal subproblem strongly convex and ensuring that the primal subproblem is decomposable over the variables; The alternating direction method of multipliers (ADMM) \cite{douglas1956numerical,eckstein1992douglas,boyd2011ADMM} simplifies the primal update by decomposing the minimization of $f$ and $h$ in MM. Other primal-dual methods include the proximal ADMM \cite{attouch2008augmented} and the method of quadratic approximation \cite{polyak1970iterative}, where the former adds a weighted proximal term to the primal update to make it separable and the latter incorporates second-order information of $f$ into the primal update of DA.

Many primal-dual algorithms for solving problem \eqref{problem} can be understood from the perspective of model-based optimization, in which the primal function or the dual function is approximated by a surrogate function (\emph{also called model}). To list a few, the update in DA takes the form of
\[v^{k+1} = v^k+\alpha \nabla q(v^k)\]
where $q(v) = \min_x F(x)+\langle v, Ax-b\rangle$ is the Lagrange dual function of problem \eqref{problem} and the update is equivalent to minimizing the model $-q(v^k)+\langle -\nabla q(v^k), v-v^k \rangle+\frac{1}{2\alpha}\|v-v^k\|^2$ of $-q(v)$. Similarly, the primal-dual gradient method \cite{alghunaim2020linear} solves problem \eqref{problem} with $h\equiv 0$ and takes the form of
\begin{align*}
    x^{k+1} &= x^k - \alpha \nabla_x L(x^k,v^k),\quad v^{k+1} = v^k + \alpha\nabla_v L(x^{k+1}, v^k)
\end{align*}
where $L(x,v) = f(x)+\langle v, Ax-b\rangle$ is the Lagrange function. For this method, the primal update minimizes the model $f(x^k)+\langle \nabla f(x^k), x-x^k\rangle+\frac{1}{2\alpha}\|x-x^k\|^2$ that approximates $f(x)$ in DA. In addition, the work \cite{polyak1970iterative} considers a second-order model that incorporates the Hessian information in the primal update of MM, and \cite{liao2025bundlebasedMM} applies a cutting-plane model in the dual update of DA.



    This paper adapts the proximal bundle model \cite{lemarechal1993convex,lemarechal1975extension,wolfe1975method} in unconstrained optimization to both the primal and dual updates in DA and MM, resulting in a family of primal-dual bundle methods (BDA and BMM). Our algorithms are motivated by the following facts: 1) DA and MM can be viewed as model-based methods with a proximal linear model (first-order Taylor expansion plus a proximal term) of the (augmented) Lagrange dual function; 2) a higher approximation accuracy of the model often leads to a faster convergence of the optimization methods \cite{asi2019importance}; and 3) the proximal bundle model yields higher approximation accuracy than the proximal linear model. Our methods adopt a \emph{general} bundle framework, which enables the use of several popular models, including the Polyak model, the two-cut model, the cutting-plane model, and the Polyak cutting-plane model. These models incorporate the problem parameters or the function value and gradient information at historical iterates and yield a higher approximation accuracy compared to those in most existing primal-dual methods. The main contributions are summarized as follows:
\begin{enumerate}[1)]
    \item We propose a family of novel primal-dual bundle methods, which adapt the proximal bundle model, typically used for unconstrained optimization, to DA and MM to accelerate their convergence.
    \item We theoretically analyse the convergence of our methods under standard assumptions that allow for a broad range of surrogate models.  
    \item We demonstrate the practical advantages of our methods by numerical experiments. Compared to alternative methods, our methods converge not only faster, but also more robustly in the parameters.
\end{enumerate}
An early effort with a similar idea is \citep{liao2025bundlebasedMM}, which applies the proximal bundle model in the dual update of DA to solve linear equality-constrained problems with the linear objective function. In contrast, our methods handle general convex objective functions, apply the proximal bundle model to both the primal and dual updates, and consider both DA and MM.

The remaining part of this paper is organized as follows: Section \ref{algorithm development} develops the algorithms and Section \ref{sec:conv_ana} analyses the convergence of the proposed methods. Section \ref{Section numerical experiment} illustrates the practical advantages of our methods through
numerical experiments, and Section \ref{sec:conclusion} concludes the paper.

\textbf{Notation and Definition.} We denote by $\langle \cdot,\cdot \rangle$ the Euclidean inner product and by $\|\cdot\|$ the Euclidean norm. For any $f:\mathbb{R}^n \rightarrow \mathbb{R}$, $\partial f(x) \subset \mathbb{R}^n$ represents its subdifferential at $x$. If $f$ is differentiable, then $\partial f(x)=\{\nabla f(x)\}$. For any $A\in \mathbb{R}^{m \times n}$, $\operatorname{Range}(A) = \{ y \in \mathbb{R}^m \mid y=Ax, \forall x \in \mathbb{R}^n\}$ and $\text{Null}(A) = \{ x \in \mathbb{R}^n \mid Ax=0\}$ represent its range and null space, respectively, and $\sigma_A$ is the smallest non-zero singular value of $A$.

For any function $f:\mathbb{R}^n \rightarrow \mathbb{R}\cup\{+\infty\}$, we say it is $\beta$-smooth for some $\beta>0$ if it is differentiable and
\[\|\nabla f(x)-\nabla f(y)\| \leq \beta\|x-y\|\quad \forall x,y \in \mathbb{R}^n\]
and it is $\mu$-strongly convex for some $\mu>0$ if
\[\langle g_x -g_y,x-y\rangle \geq \mu \|x-y\|^2\quad \forall x,y \in \mathbb{R}^n, ~ g_x \in \partial f(x), ~ g_y \in \partial f(y).\]

\section{Algorithm Development}\label{algorithm development}

This section extends the proximal bundle model that is typical in unconstrained optimization to accelerate the dual ascent method (DA) and the method of multipliers (MM) \cite{boyd2011ADMM}.

\subsection{Preliminaries: DA, MM, and the proximal bundle model}
\subsubsection{DA and MM}\label{review of DA MM...}

DA and MM are fundamental algorithms for solving linear equality-constrained problems, and several primal-dual methods can be viewed as their variants, such as the primal-dual gradient method \cite{alghunaim2020linear}, the linearized method of multipliers (linearized MM) \cite{xu2021first}, and the alternating direction method of multipliers (ADMM) \cite{boyd2011ADMM}. DA and MM for solving problem \eqref{problem} take the following forms.

\begin{itemize}
\item Dual Ascent (DA): At each iteration $k\ge 0$, 
\begin{align}
x^{k+1} &= \underset{x}{\operatorname{\arg\;\min}}~ L(x,v^k)\label{eq:DA_x}\\
v^{k+1} &= v^k + \alpha (Ax^{k+1}-b)\label{eq:DA_v}
\end{align}
    where $L(x,v)=F(x)+\langle v,Ax-b\rangle$ and the step-size $\alpha > 0$. DA applies steepest descent \cite{boyd2004convex} to minimizing the negative dual function and requires the strong convexity of $F$ to guarantee the convergence.

\item Method of Multipliers (MM): For some $\rho>0$ and at each iteration $k\ge 0$,
\begin{align}
    x^{k+1} &= \underset{x}{\operatorname{\arg\;\min}}~ L_\rho(x,v^k)\label{eq:MM_x}\\
    v^{k+1} &= v^k + \rho (Ax^{k+1}-b)\label{eq:MM_v}
\end{align}
where $L_\rho(x,v)=F(x)+ \langle v,Ax-b\rangle + \frac{\rho}{2}\|Ax-b\|^2$. MM is equivalent to applying the proximal point method \cite{Rocka1976proximalpoint}  to solving the dual problem of \eqref{problem}. Compared with DA, MM only
requires the objective function to be convex, rather than strongly convex, which significantly broadens the range of solvable problems.
\end{itemize}

Both DA and MM can be viewed from the perspective of model-based optimization. To see this, note that the update of DA and MM are equivalent to steepest descent on minimizing $-q(v)$ and $-q_\rho(v)$, respectively, where
\begin{align}\label{eq:dual_func}
    q(v) & = \min_x~L(x,v),\qquad q_\rho(v) = \min_x~L_\rho(x,v).
\end{align}
DA and MM can also be expressed as, for some $c_d > 0$,
\begin{align}
    v^{k+1} &= \underset{v}{\operatorname{\arg\;\min}} -q(v^k)-\langle \nabla q(v^k), v-v^k\rangle+\frac{c_d}{2}\|v-v^k\|^2\label{eq:DA_dual_model}\\
    v^{k+1} &= \underset{v}{\operatorname{\arg\;\min}} -q_\rho(v^k)-\langle \nabla q_\rho(v^k), v-v^k\rangle+\frac{c_d}{2}\|v-v^k\|^2\label{eq:MM_dual_model}
\end{align}
respectively, where the minimized proximal linear function is a surrogate function (\emph{also called model}) of $q$ in DA and $q_\rho$ in MM. The proximal linear surrogate function has a simple form, but it can lead to low approximation accuracy, which further slows down the convergence of the algorithms.

\subsubsection{The proximal bundle model}

The proximal bundle model is adopted in bundle methods \cite{lemarechal1993convex,lemarechal1975extension,wolfe1975method} for solving unconstrained optimization problems. At each iteration $k\ge 0$, the proximal bundle model of a function $\phi:\mathbb{R}^n\rightarrow\mathbb{R}\cup\{+\infty\}$ takes the form of 
\begin{equation}\label{eq:bundle}
    \hat{\phi}^k(x) + \frac{c}{2}\|x-x^k\|^2
\end{equation}
where $\hat{\phi}^k$ is an approximation (often a minorant) of $\phi$, $c\ge 0$ is a penalty parameter, and $x^k$ is the iterate at the $k$th iteration.

The proximal bundle model is explicitly or implicitly employed in many optimization methods. For example, when minimizing a function $\phi$, steepest descent and the proximal point method correspond to \eqref{eq:bundle} with $\hat{\phi}^k(x) = \phi(x^k)+\langle \nabla \phi(x^k), x-x^k\rangle$ and $\hat{\phi}^k(x)=\phi(x)$, respectively, the proximal bundle method in \cite{lemarechal1993convex} corresponds to \eqref{eq:bundle} with the cutting-plane model \cite{kelley1960cutting} where
\[\hat{\phi}^k(x) = \max_{t\in S^k}~\phi(x^t)+\langle\nabla \phi(x^t),x-x^t\rangle\]
for an index set $S^k\subseteq \{0,1,\ldots,k\}$. Other variants of bundle methods include the trust region-based bundle method and the relaxation-based bundle method \cite[Sec. XV.]{lemarechal1993convex}, which incorporates the cutting-plane model with the trust-region technique and a relaxation technique, respectively. The use of multiple cutting planes often yields a high approximation accuracy of $\hat{\phi}^k$ on $\phi$, and the resulting proximal bundle method not only converges fast but also exhibits extraordinary robustness in the parameter $c$. Due to this reason, a growing body of works extends the bundle method to solve optimization problems with different structures. For example, \cite{cederberg2025an,fischer2025asynchronous} propose asynchronous bundle methods to solve distributed optimization problems, \cite{liang2024single} develops a stochastic
composite proximal bundle method to solve stochastic optimization with composite objective functions, and \cite{hare2010redistributed,de2019proximal} apply the bundle method to solve non-convex problems with certain structures.

\subsection{Algorithm}

Inspired by the generality, quick convergence, and robustness of optimization methods with the proximal bundle model \eqref{eq:bundle}, we adapt it to the primal and dual updates in DA and MM, respectively.
\begin{itemize}
    \item \textbf{Dual bundle update}: By \eqref{eq:DA_dual_model}--\eqref{eq:MM_dual_model}, DA and MM take the form of the proximal bundle method \eqref{eq:bundle} where $\phi$ is replaced with $-q$ and $-q_\rho$, respectively. As a result, the straightforward bundle extensions of DA and MM are
    \begin{align}
        v^{k+1} = \underset{v}{\operatorname{\arg\;\min}} -\hat{q}^k(v)+\frac{c_d}{2}\|v-v^k\|^2\label{eq:BDA-D}\\
                v^{k+1} = \underset{v}{\operatorname{\arg\;\min}}-\hat{q}_\rho^k(v)+\frac{c_d}{2}\|v-v^k\|^2\label{eq:BMM-D}
    \end{align}
    respectively, where $\hat{q}^k$ and $\hat{q}_\rho^k$ are models of $q$ and $q_\rho$, respectively, and $c_d\ge 0$. Incorporating standard surrogate models such as the cutting-plane model into \eqref{eq:BDA-D} and \eqref{eq:BMM-D}, we can obtain novel variants of DA and MM.
    \item \textbf{Primal bundle update}: The models $\hat{q}^k$ and $\hat{q}_\rho^k$ in the dual bundle updates \eqref{eq:BDA-D}--\eqref{eq:BMM-D} often involve the function values or the gradients of $q$ and $q_\rho$ at $v_k$, which requires solving \eqref{eq:DA_x} and \eqref{eq:MM_x}. To reduce the computational cost, we adapt the bundle model to the primal updates \eqref{eq:DA_x} and \eqref{eq:MM_x}, leading to
    \begin{align}
        x^{k+1} &= \underset{x}{\operatorname{\arg\;\min}}~ \hat{F}^k(x)+\frac{c_p}{2}\|x-x^k\|^2+\langle v^k, Ax-b\rangle\label{eq:BDA-x}\\
        x^{k+1} &= \underset{x}{\operatorname{\arg\;\min}} ~\hat{F}^k(x)+\frac{c_p}{2}\|x-x^k\|^2+\langle v^k, Ax-b\rangle+\frac{\rho}{2}\|Ax-b\|^2\label{eq:BMM-x}
    \end{align}
    where $\hat{F}^k$ is a surrogate model of $F$.
\end{itemize}
We refer to the algorithm with the primal update \eqref{eq:BDA-x} and the dual update \eqref{eq:BDA-D} as the bundle dual ascent method (\textbf{BDA}), and the one with \eqref{eq:BMM-x} and \eqref{eq:BMM-D} as the bundle method of multipliers (\textbf{BMM}), which generalize several existing methods. In particular, when $h\equiv 0$, $\hat{F}^k(x)=f(x^k)+\langle \nabla f(x^k), x-x^k\rangle$, and $\hat{q}^k(v)=F(x^{k+1})+\langle Ax^{k+1}-b, v\rangle$, BDA reduces to the primal-dual gradient method \cite{alghunaim2020linear}, and when $\hat{F}^k(x)=f(x^k)+\langle \nabla f(x^k), x-x^k\rangle$ and $\hat{q}_\rho^k(v)=F(x^{k+1})+\langle Ax^{k+1}-b, v\rangle + \frac{\rho}{2}\|Ax^{k+1}-b\|^2$, BMM becomes the linearized MM \cite{xu2017accelerated,xu2021first}.

By letting the primal and dual models take specific forms in BDA and BMM, we obtain several variants. For BDA, we have
\begin{equation*}
    \begin{cases}
        \textbf{BDA-P:}\qquad &\text{primal update \eqref{eq:BDA-x} in BDA~}~+~ \text{dual update \eqref{eq:DA_v} in DA}\\
        \textbf{BDA-D:}\qquad &\text{primal update \eqref{eq:DA_x} in DA~}~+~ \text{dual update \eqref{eq:BDA-D} in BDA}
    \end{cases}
\end{equation*}
and for BMM, we have
\begin{equation*}
    \begin{cases}
        \textbf{BMM-P:}\qquad &\text{primal update \eqref{eq:BMM-x} in BMM~} ~+~ \text{dual update \eqref{eq:MM_v} in MM}\\
        \textbf{BMM-D:}\qquad &\text{primal update \eqref{eq:MM_x} in MM} ~+~ \text{dual update \eqref{eq:BMM-D} in BMM.}
    \end{cases}
\end{equation*}
Specifically, the -P methods adopt the same primal update as BDA and BMM and the same dual updates as DA and MM. The -D methods take the same primal updates as DA and MM, and the same dual update as BDA and BMM.

\begin{figure}[t]
\centering
    \begin{subfigure}[t]{0.26\linewidth}
        \includegraphics[width=\linewidth]{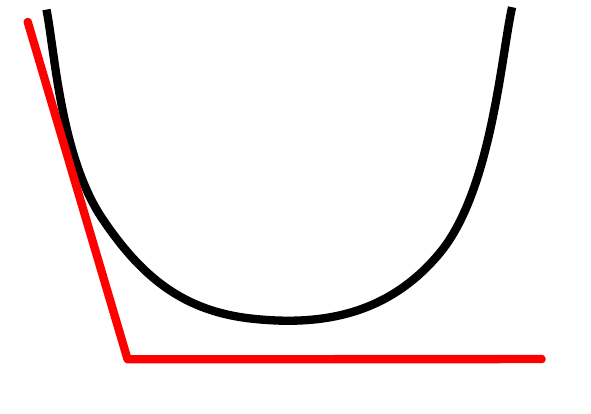}
        \caption{Polyak}
    \end{subfigure} 
    \hfill
    \begin{subfigure}[t]{0.26\linewidth}
        \includegraphics[width=\linewidth]{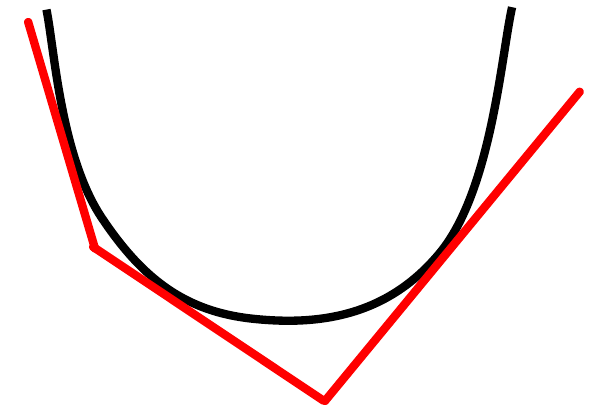}
        \caption{Cutting-plane}
    \end{subfigure} 
    \hfill
    \begin{subfigure}[t]{0.26\linewidth}
        \includegraphics[width=\linewidth]{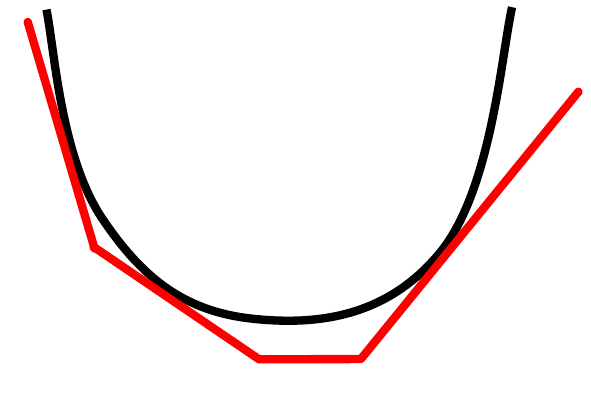}
        \caption{Polyak cutting-plane}
    \end{subfigure} 
    \caption{Surrogate functions in the Polyak model \eqref{eq:pol_model}, cutting-plane model \eqref{eq:cpm} and the Polyak cutting-plane model \eqref{eq:pcpm}}
    \label{fig:surrogate functions}
\end{figure}

\subsubsection{Candidates of the primal model and the dual model}\label{sssec:mod_cand}

We provide several choices for the primal and dual models in BDA and BMM. For the primal objective function $F$, we consider the primal model of the form $\hat{F}^k(x) = \hat{f}^k(x)+h(x)$, where $\hat{f}^k$ is a surrogate model of $f$ and a natural choice is $\hat{f}^k=f$. Below, we first provide four models for $f$, and then adapt them to the dual functions $q$ and $q_\rho$.

\bigskip

\textbf{Primal models for $f$}. We focus on convex models that take the maximum of affine functions, where each affine function is a minorant of $f$. Consequently, these models are also minorants of $f$. We visualize three of these models in Fig. \ref{fig:surrogate functions}.
    \begin{itemize}
        \item Polyak model: The model originates from the Polyak step-size \cite{polyak1987introduction} for steepest descent, which takes the following form when minimizing a function $f$:
    \[x^{k+1} = x^k - \alpha^k\nabla f(x^k),\qquad \alpha^k = \frac{f(x^k)-\ell_f}{\|\nabla f(x^k)\|^2}\]
    where $\ell_f$ is a lower bound of $\min_x f(x)$. In the update, $x^{k+1}$ is a minimizer of
    \begin{equation}\label{eq:pol_model}
        \hat{f}^k(x) = \max\{f(x^k)+\langle\nabla f(x^k), x-x^k\rangle, \ell_f\}
    \end{equation}
    which is referred to as the Polyak model of $f$.
        \item Cutting-plane model: It takes the maximum of historical cutting planes \cite{kelley1960cutting}:
        \begin{equation}\label{eq:cpm}
            \hat{f}^k (x) = \max_{t\in S^k}~f(x^t)+\langle \nabla f(x^t),x-x^t\rangle
        \end{equation}
        where $S^k\subseteq \{0,1,\ldots,k\}$ is a subset of historical iteration indexes.
        \item Polyak cutting-plane model: The model is designed as the maximum of the Polyak model and the cutting-plane model:
        \begin{equation}\label{eq:pcpm}
            \hat{f}^k (x) = \max\{\ell_f, \max_{t\in S^k}~f(x^t)+\langle \nabla f(x^t),x-x^t\rangle\}.
        \end{equation}
        \item Two-cut model: It takes the maximum of the cutting planes of $f$ and $\hat{f}^{k-1}$ at $x^k$:
        \begin{equation}\label{eq:two-cut}
            \hat{f}^k (x) = \max\{\hat{f}^{k-1}(x^k)+\langle \hat{g}^k, x-x^k\rangle,~f(x^k)+\langle \nabla f(x^k),x-x^k\rangle\}
        \end{equation}
        where $\hat{g}^k\in \partial \hat{f}^{k-1}(x^k)$ and $\hat{f}^0(x) = f(x^0)+\langle\nabla f(x^0), x-x^0\rangle$.
    \end{itemize}
    \begin{remark}
        All four models \eqref{eq:pol_model}--\eqref{eq:two-cut} take the maximum of minorants of $f$ and they are all majorants of the linear model $f(x^k)+\langle\nabla f(x^k), x-x^k\rangle$. Therefore, they achieve higher accuracy in approximating $f$ compared to the linear model
        \[f(x)-\hat{f}^k \le f(x) - (f(x^k)+\langle\nabla f(x^k), x-x^k\rangle).\]
        Similarly, the Polyak cutting-plane model yields higher approximation accuracy than the Polyak model and the cutting-plane model (see Fig. \ref{fig:surrogate functions}).
    \end{remark}

\textbf{Dual model for $q$.} Adapting the four models to $q$ requires both the function value and the gradient of $q$ at $v^k$, which requires solving \eqref{eq:DA_x} exactly. However, the primal update \eqref{eq:BDA-x} only offers an approximate solution. To address this issue, we propose to approximate $q(v^k)$ and its gradient using $x^{k+1}$ generated by the primal update of BDA. When $x^{k+1}\in \operatorname{\arg\;\min}_x~L(x,v^k)$, it follows that
\begin{equation}\label{eq:q_and_graq}
    q(v^k) = L(x^{k+1},v^k),\qquad \nabla q(v^k) = \nabla_v L(x^{k+1},v^k)
\end{equation}
and the cutting plane of $q$ at $v^k$ is
\[L(x^{k+1},v^k)+\langle \nabla_v L(x^{k+1},v^k), v-v^k\rangle.\]
Inspired by this, we construct the following approximate cutting plane of $q$ at $v^k$:
\begin{equation}\label{eq:Cqk}
    C_q^k(v) = L(x^{k+1},v^k)+\langle \nabla_v L(x^{k+1},v^k), v-v^k\rangle
\end{equation}
where $x^{k+1}$ is computed from \eqref{eq:BDA-x}. Let $u_q$ be a known upper bound of $\max_v q(v)$. Incorporating \eqref{eq:Cqk} into the four models \eqref{eq:pol_model}--\eqref{eq:two-cut}, we obtain
\begin{equation}\label{eq:model_q}
    \hat{q}^k(v) = \begin{cases}
        \min\{C_q^k(v), u_q\}, & \text{Polyak model}\\
        \min_{t\in S^k}~C_q^t(v), & \text{Cutting-plane model}\\
        \min\{u_q, \min_{t\in S^k}~C_q^t(v)\}, & \text{Polyak cutting-plane model}\\
        \min\{\hat{q}^{k-1}(v^k)+\langle \hat{g}_q^k, v-v^k\rangle , C_q^k(v)\}, & \text{Two-cut model}
    \end{cases}
\end{equation}
where $-\hat{g}_q^k\in \partial(-\hat{q}^{k-1})(v^k)$ and $\hat{q}^0(v)=C_q^0(v)$ in the two-cut model.
\bigskip

\textbf{Dual model for $q_\rho$}. We adapt the four models \eqref{eq:pol_model}--\eqref{eq:two-cut} to the augmented Lagrange dual function $q_\rho$ in the same way of adapting them to $q$. Similar to \eqref{eq:Cqk}, we construct the following approximate cutting-plane of $q_\rho$ at $v^k$ using $x^{k+1}$:
\begin{equation}\label{eq:Cq_rhok}
    C_{q,\rho}^k(v) = L_\rho(x^{k+1},v^k)+\langle \nabla_v L_\rho(x^{k+1},v^k), v-v^k\rangle
\end{equation}
where $L_\rho$ is the augmented Lagrangian. Let $u_{q,\rho}$ be a known upper bound of $\max_v q_\rho(v)$. Incorporating \eqref{eq:Cq_rhok} into the four models \eqref{eq:pol_model}--\eqref{eq:two-cut}, we obtain
\begin{equation}\label{eq:model_qrho}
    \hat{q}_\rho^k(v) = \begin{cases}
        \min\{C_{q,\rho}^k(v), u_{q,\rho}\}, & \text{Polyak model}\\
        \min_{t\in S^k}~C_{q,\rho}^t(v), & \text{Cutting-plane model}\\
        \min\{u_{q,\rho},\min_{t\in S^k}~C_{q,\rho}^t(v)\}, & \text{Polyak cutting-plane model}\\
        \min\{\hat{q}_\rho^{k-1}(v^k)+\langle \hat{g}_{q_\rho}^k, v-v^k\rangle, C_{q,\rho}^k(v)\}, & \text{Two-cut model}
    \end{cases}
\end{equation}
where $-\hat{g}_{q_\rho}^k\in \partial(-\hat{q}_\rho^{k-1})(v^k)$ and $\hat{q}_\rho^0(v)=C_{q,\rho}^0(v)$ in the two-cut model.

\subsubsection{Solving the subproblems in BDA and BMM}

When using the models in Section \ref{sssec:mod_cand}, the subproblems in BDA and BMM can be solved efficiently. To see this, note that with the models in Section \ref{sssec:mod_cand}, the subproblems in the primal updates of BDA and BMM can be simplified as
\begin{equation}\label{eq:simplified_primal_prob}
    \underset{x}{\operatorname{minimize}}~\max_{i\in\{1,\ldots,M\}}\{\tilde{a}_i^Tx+\tilde{b}_i\}+\frac{x^TCx}{2}+d^Tx+h(x)
\end{equation}
where $M$ is the number of affine functions in $\hat{f}^k$. Take cutting-plane model \eqref{eq:cpm} for example,
\begin{align*}
    & \tilde{a}_i = \nabla f(x^i),~~~~~ \tilde{b}_i = f(x^i) -\langle \nabla f(x^i),x^i \rangle\\
    & d = A^Tv^k-c_px^k,~~C=c_pI ~~~~\text{for BDA}\\
    & d = A^Tv^k-c_px^k-\rho A^Tb,~~C=c_pI+\rho A^TA ~~~~\text{for BMM}.
\end{align*}
Problem \eqref{eq:simplified_primal_prob} can be equivalently transformed into a linearly constrained problem
\begin{equation}\label{eq:QP}
    \begin{split}
        \underset{x,t}{\operatorname{minimize}}~~&~t+\frac{x^TCx}{2}+d^Tx+h(x)\\
        \operatorname{subject~to}~ &~\tilde{a}_i^Tx+\tilde{b}_i \le t,~i\in\{1,\ldots,M\}
    \end{split}
\end{equation}
which reduces to an easy-to-solve convex quadratic program when $h\equiv 0$. 

For problem \eqref{eq:QP}, it is preferable to solve it by dual approaches since the dimension $M$ of the dual variable is usually small (e.g., $M=5,10$) in practical implementations. To formulate the dual problem, we define
\[\tilde{A}=(\tilde{a}_1, \ldots, \tilde{a}_m)^T, \quad \tilde{b} = (\tilde{b}_1;\ldots;\tilde{b}_m)\]
and rewrite the constraint in \eqref{eq:QP} as
\[\tilde{A}x+\tilde{b}\le t\mb{1}.\]

\begin{lemma}\label{lem:subproblem_dual}
The Lagrange dual of problem \eqref{eq:QP} is
\begin{equation}\label{eq:subproblem_dual}
    \begin{split}
        \underset{\lambda\in\R^M}{\operatorname{maximize}}  ~~~& g(\lambda)\\
        \st ~~& \mb{1}^T\lambda =1 , ~\lambda \geq 0,
    \end{split}
\end{equation}
where
\begin{equation*}
    \begin{split}
        g(\lambda) 
        =&\inf_x \left\{ h(x) +\frac{1}{2}x^TCx +d^Tx + \lambda^T\tilde{A}x \right\}+\tilde{b}^T\lambda
    \end{split}
\end{equation*}
is the dual function. Letting $x_{\lambda} = \operatorname{\arg\;\min}_x h(x) +\frac{1}{2}x^TCx +d^Tx + \lambda^T\tilde{A}x$, it holds that
\[
\nabla g(\lambda) = \tilde{A}x_\lambda+b.\]
Furthermore, if $\lambda^\star$ is optimal to \eqref{eq:subproblem_dual}, then the unique optimal solution of problem \eqref{eq:simplified_primal_prob} is $x_{\lambda^\star}$.
\end{lemma}
\begin{proof}
    See Appendix \ref{appen:subproblem_dual}.
\end{proof}
Since $C$ is positive definite, the dual function $g$ is smooth, and we can solve problem \eqref{eq:subproblem_dual} by dual gradient-type methods such as FISTA \cite{Beck09FISTA}, which typically only takes a few iterations to solve problems with low variable dimension. In the implementation of FISTA, the gradient $\nabla g(\lambda)=\tilde{A}x_{\lambda}+b$ and the projection onto the constraint set (simplex) of \eqref{eq:subproblem_dual} can be implemented in the complexity of $O(M)$ \cite{condat2016fast}. For the subproblem in the primal update \eqref{eq:BDA-x} of BDA, $C=c_p I$ and \[x_{\lambda}=\prox_{h/c_p}(-(d+\tilde{A}^T\lambda)/c_p)+\tilde{b},\]
which is easy to solve if $h=\|\cdot\|_1$ or $h$ is the indicator function of an easy-to-project constraint set. Specifically, when $h=\|\cdot\|_1$, $n=100,000$, and $M=15$, applying FISTA to solve \eqref{eq:subproblem_dual} only takes $19 \sim 40$ iterations and $0.1\sim 0.2$ second in a PC with the Apple M3 8-core CPU. For the subproblem in the primal update \eqref{eq:BMM-x}, $C=c_pI+\rho A^TA$ and the computation of $x_{\lambda}$ requires the minimization of a strongly convex objective function.

The subproblems in the dual updates \eqref{eq:BDA-D} and \eqref{eq:BMM-D} of BDA and BMM take similar forms as \eqref{eq:simplified_primal_prob} with $x$ replaced with $v$, $C=c_dI$, and $h\equiv 0$, and can be rewritten as a quadratic program following the reformulation procedure from \eqref{eq:simplified_primal_prob} to \eqref{eq:QP}. The resulting quadratic program can be solved by either the primal or the dual approaches efficiently due to its simple form.

\section{Convergence Analysis}\label{sec:conv_ana}

This section analyses the convergence of BDA and BMM.

\subsection{Assumptions and assisting lemmas}\label{ssec:assum_assist_lemmas}

To analyse the convergence, we first impose assumptions on problem \eqref{problem} and the primal and dual models, and derive assisting lemmas.
\begin{assumption}\label{asm:prob}
The following holds:
     \begin{enumerate}[(a)]
     \item $f: \mathbb{R}^n \rightarrow \mathbb{R}$ is proper, closed, convex, and $\beta$-smooth for some $\beta>0$.
     \item $h: \mathbb{R}^n \rightarrow \mathbb{R}\cup\{+\infty\}$ is proper, closed, and convex.
     \item There exists at least one optimal solution $x^\star$ to problem \eqref{problem}.
    \end{enumerate}
\end{assumption}

Assumption \ref{asm:prob} is standard in convex optimization, which guarantees the strong duality between problem \eqref{problem} and its Lagrange dual problem \cite[Proposition 6.2.1]{bertsekas1997nonlinear}. Under Assumption \ref{asm:prob} and by the KKT condition \cite{boyd2004convex}, there exists $v^\star\in\mathbb{R}^m$ such that
\begin{align}
    & \mathbf{0} \in \partial F(x^\star) + A^Tv^\star\label{KKT2}
\end{align}
where $v^\star$ is a dual optimum to problem \eqref{problem}. The following lemma guarantees the existence of a particular dual optimum.
\begin{lemma}\label{lemma:vstar_in_R(A)}
    Suppose that Assumption \ref{asm:prob} holds. Then, there exists a dual variable $v^\star$ such that \eqref{KKT2} holds and $v^\star-v^0\in \operatorname{Range}(A)$. 
\end{lemma}
\begin{proof}
See Appendix \ref{appen:prop_vstar_in_R(A)}.
\end{proof}

\begin{assumption}\label{asm:pri_model}
    \textbf{(Primal model)} The primal model $\hat{f}^k$ satisfies 
    \begin{enumerate}[(a)]
    \item $\hat{f}^k$ is convex. \label{primal model assumption a}
    \item $\hat {f}^k(x) \geq f(x^k) + \langle \nabla f(x^k),x-x^k\rangle$ for all $x\in\R^n$. \label{primal model assumption b}
    \item $\hat{f}^k(x)\leq f(x)$ for all $x\in\R^n$.\label{primal model assumption c}
    \end{enumerate}
\end{assumption}

Assumption \ref{asm:pri_model} requires $\hat{f}^k$ to be a convex minorant of $f$ and a majorant of the cutting plane $f(x^k)+\langle\nabla f(x^k),x-x^k\rangle$, which is standard in bundle methods \cite{mateo2023optimal,liao2025bundlebasedMM}. Moreover, by letting $x=x^k$ in Assumption \ref{asm:pri_model} (b)--(c), we have $\hat{f}^k(x^k)=f(x^k)$, i.e., the model $\hat{f}^k$ \emph{accurately} evaluates $f$ at $x^k$. Under Assumptions \ref{asm:prob}--\ref{asm:pri_model}, the following lemma bounds the approximation error of $\hat{f}^k$ on $f$ and its subgradients on $\nabla f$, respectively.
\begin{lemma}\label{lemma:smooth_f_hatf}
    Suppose that Assumptions \ref{asm:prob}--\ref{asm:pri_model} hold. Then, for each $k\ge 0$, 
    \begin{align}
        & f(x) \leq \hat{f}^k(x) + \frac{\beta}{2}\|x-x^k\|^2\label{lemma 3.1 1}\\
        & \|\hat{g} - \nabla f(x)\| \leq \beta\|x-x^k\|,\forall \hat{g} \in \partial \hat{f}^k(x). \label{like smooth}
    \end{align}
   
\end{lemma}
\begin{proof}
See Appendix \ref{appen:lem_smooth_f_hatf}. 
\end{proof}
Lemma \ref{lemma:smooth_f_hatf} shows that the approximation error in both the function value and the gradient can be bounded by a function of the distance $\|x-x^k\|$, which is vital to our convergence analysis.

\begin{assumption}\label{asm:dual_modelDA}
    \textbf{(Dual model in BDA)} The dual model $\hat{q}^k$ satisfies 
    \begin{enumerate}[(a)]
    \item $\hat{q}^k$ is concave.
    \item $\hat{q}^k(v) \leq L(x^{k+1},v)$ for all $v\in\mathbb{R}^m$.
    \item $\hat{q}^k(v)\geq q(v)$  for all $v\in\mathbb{R}^m$.
    \end{enumerate}
\end{assumption}    

\begin{assumption}\label{asm:dual_modelMM}
    \textbf{(Dual model in BMM)} The dual model $\hat{q}_\rho^k$ satisfies 
    \begin{enumerate}[(a)]
    \item $\hat{q}^k_\rho$ is concave.
    \item $\hat{q}^k_\rho(v) \leq L_\rho(x^{k+1}, v)$ for all $v\in\mathbb{R}^m$.
    \item $\hat{q}_\rho^k(v)\geq q_\rho(v)$ for all $v\in\mathbb{R}^m$. 
    \end{enumerate}
\end{assumption} 

 Assumptions \ref{asm:dual_modelDA}--\ref{asm:dual_modelMM} can be viewed as variants of Assumption \ref{asm:pri_model}, with $f$ replaced with $-q$ and $-q_\rho$ and $\hat{f}^k$ replaced with $-\hat{q}^k$ and $-\hat{q}_\rho^k$, respectively. To see this, note that in Assumption \ref{asm:dual_modelDA}(b) and Assumption \ref{asm:dual_modelMM}(b),
\begin{align}
    L(x^{k+1},v) &= L(x^{k+1},v^k)+\langle \nabla_v L(x^{k+1}),v-v^k\rangle\label{eq:cut_Lxv}\\
    L_\rho(x^{k+1},v) &= L_\rho(x^{k+1},v^k)+\langle \nabla_v L_\rho(x^{k+1}),v-v^k\rangle.\label{eq:cut_Lrho_xv}
\end{align}
If $x^{k+1}\in \operatorname{\arg\;\min}_x L(x,v^k)$ in Assumption \ref{asm:dual_modelDA} and  $x^{k+1}\in \operatorname{\arg\;\min}_x L_\rho(x,v^k)$ in Assumption \ref{asm:dual_modelMM}, then Assumption \ref{asm:dual_modelDA}(b) and Assumption \ref{asm:dual_modelMM}(b) reduce to
\begin{align}
    -\hat{q}^k(v) \ge -q(v)+\langle -\nabla q(v^k), v-v^k\rangle,~~ -\hat{q}_\rho^k(v) \ge -q_\rho(v)+\langle -\nabla q_\rho(v^k), v-v^k\rangle.
\end{align}
If $\hat{f}^k\ne f$, then, in general, $x^{k+1}$ is not a minimizer of $L(x,v^k)$ in BDA and $L_\rho(x,v^k)$ in BMM, and Assumption \ref{asm:dual_modelDA}--\ref{asm:dual_modelMM} only approximate Assumption \ref{asm:pri_model}. The following lemma shows that the models in Section \ref{sssec:mod_cand} satisfy Assumptions \ref{asm:pri_model}--\ref{asm:dual_modelMM}.

\begin{lemma}\label{lemma:mod_sat_asm}
    The four models \eqref{eq:pol_model}--\eqref{eq:two-cut} and $\hat{f}^k=f$ satisfy Assumption \ref{asm:pri_model}, and their extensions \eqref{eq:model_q} and \eqref{eq:model_qrho} satisfy Assumption \ref{asm:dual_modelDA} and Assumption \ref{asm:dual_modelMM}, respectively.
\end{lemma}
\begin{proof}
    See Appendix \ref{appen:lem_mod_sat_asm}.
\end{proof}





Another favorable property of the dual models \eqref{eq:model_q} and \eqref{eq:model_qrho} is that for any $k\ge 0$ and $v\in\mathbb{R}^m$,
\begin{equation}\label{eq:partial_in_range_A}
    \partial (-\hat{q}^k(v))\subseteq \operatorname{Range}(A),\qquad \partial (-\hat{q}_\rho^k(v))\subseteq \operatorname{Range}(A)
\end{equation}
which will be used in Section \ref{ssec:conv_BDA} and Section \ref{ssec:conv_BMM} to derive linear convergence rates of our methods and is assumed below.
\begin{assumption}\label{asm:dual subgradient in range a}
    Equation \eqref{eq:partial_in_range_A} holds for all $k\ge 0$ and $v\in\mathbb{R}^m$.
\end{assumption}

\begin{lemma}\label{lem:dual_model_subdifferential}
     The dual models \eqref{eq:model_q} and \eqref{eq:model_qrho} satisfy Assumption \ref{asm:dual subgradient in range a}.
\end{lemma}
\begin{proof}
    See Appendix \ref{appen:lem_dual_model_subdifferential}.
\end{proof}

\subsection{Convergence of BDA}\label{ssec:conv_BDA}

With the assisting lemmas in Section \ref{ssec:assum_assist_lemmas}, we derive the convergence of BDA. To this end, define
\begin{align*}
    z^k = 
\begin{bmatrix}
x^k\\
v^k 
\end{bmatrix},\quad z^\star = 
\begin{bmatrix}
x^\star\\
v^\star
\end{bmatrix},\quad \Lambda = \begin{bmatrix}
\frac{c_p}{2} & 0\\
0 & \frac{c_d}{2}  
\end{bmatrix}
\end{align*}
where $x^\star$ is optimal to problem \eqref{problem} and $v^\star$ is a dual optimum satisfying $v^\star-v^0\in \operatorname{Range}(A)$ whose existence is guaranteed by Lemma \ref{lemma:vstar_in_R(A)}.

\begin{theorem}\label{thm:BDA}
\textbf{(Convergence of BDA)} Suppose that $f$ is $\mu$-strongly convex for some $\mu>0$ and that Assumptions \ref{asm:prob} -- \ref{asm:dual_modelDA} hold. Let $\{x^k\}$ be generated by BDA (\eqref{eq:BDA-x} and \eqref{eq:BDA-D}). If
\begin{equation}\label{eq:step_cond_BDA}
    c_p>\frac{2\beta^2}{\mu},\quad c_d > \frac{2\|A\|^2}{\mu}
\end{equation}
then $\lim_{k\rightarrow+\infty} x^k=x^\star$ and for all $k\ge 1$, 
\begin{equation}\label{eq:BDA_rate}
    \frac{1}{k}\sum_{t=1}^k \|x^\star-x^{t}\|^2=O\left(\frac{1}{k}\right), \qquad\min_{t\leq k}\|x^\star-x^t\|^2 = o\left(\frac{1}{k}\right).
\end{equation}
Furthermore, if $h(x) \equiv0$ and Assumption \ref{asm:dual subgradient in range a} holds, then
\begin{equation}\label{eq:linear_rate_BDA}
    \|z^\star-z^{k+1}\|^2_\Lambda \leq (1-\alpha)\|z^\star-z^{k}\|^2_\Lambda,
\end{equation}
where $\alpha =  \frac{\sigma_A^2}{6c_dc_p} \cdot \min\left\{1-2\beta^2/(\mu c_p), 1-2\|A\|^2/(\mu c_d)\right\}\in(0,1)$ with $\sigma_A$ being the smallest non-zero singular value of $A$.
\end{theorem}
\begin{proof}
See Appendix \ref{appen:thm_BDA}.
\end{proof}

The next theorem gives the convergence result of BDA-D where $\hat{f}^k=f$. Due to the condition $c_p>\frac{2\beta^2}{\mu}$, the convergence results in Theorem \ref{thm:BDA} do not cover BDA-D where $c_p=0$.

\begin{theorem}\label{thm:BDA-D}
    \textbf{(Convergence of BDA-D)} Suppose that $f$ is $\mu$-strongly convex and that Assumptions \ref{asm:prob} and \ref{asm:dual_modelDA} hold. Let $\{x^k\}$ be generated by BDA-D (\eqref{eq:DA_x} and \eqref{eq:BDA-D}). If $c_d>\frac{\|A\|^2}{\mu}$, then $\lim_{k\rightarrow+\infty} x^k=x^\star$ and \eqref{eq:BDA_rate} holds. If, in addition, $h(x) \equiv0$ and Assumption \ref{asm:dual subgradient in range a} holds, then
    \begin{align}
        & \|x^\star-x^{k+1}\|^2 \le  \frac{c_d}{\mu-\|A\|^2/c_d}(1-\alpha)^k\|v^\star-v^{0}\|^2\label{eq:BDA-D_x_rate}\\
        & \|v^\star-v^k\|^2 \le (1-\alpha)^k\|v^\star-v^{0}\|^2\label{eq:BDA-D_v_rate}
    \end{align}
    where $\alpha = \frac{\sigma_A^2}{\beta^2 c_d}\cdot \left(\mu-\frac{\|A\|^2}{c_d}\right)\in (0,1)$.
\end{theorem}

\begin{proof}
 See Appendix \ref{appen:thm_BDA_D}.
\end{proof}

\begin{remark}
    BDA and BDA-D generalize the primal-dual gradient method and DA, respectively, and the convergence rates in Theorems \ref{thm:BDA}--\ref{thm:BDA-D} are of the same order as the primal-dual gradient method \cite{alghunaim2020linear} and DA \cite{Lu16}. However, due to the use of more accurate models, BDA and BDA-D can yield much faster convergence, which will be shown numerically in Section \ref{Section numerical experiment}.
\end{remark}

\subsection{Convergence of BMM}\label{ssec:conv_BMM}

This subsection analyses the convergence of BMM. By the KKT condition of problem \eqref{problem}, a point $x^\star$ is optimal to problem \eqref{problem} if and only if it is feasible and \eqref{KKT2} holds for some $v^\star$, which is equivalent to 
\begin{equation}\label{eq:KKT_BMM}
    g_F^\star \in \operatorname{Range}(A^T),\qquad Ax^\star=b
\end{equation}
for some $g_F^\star \in \partial F(x^\star)$. Moreover, letting $P_{N}$ be the projection matrix onto the space $\operatorname{Null}(A)$, \eqref{eq:KKT_BMM} can be rewritten as
\begin{equation}\label{eq:KKT_BMM2}
    P_{N}g_F^\star=\mathbf{0},\qquad Ax^\star=b.
\end{equation}
\begin{theorem}\label{thm:BMM}
    \textbf{(Convergence of BMM)} Suppose that Assumptions \ref{asm:prob}, \ref{asm:pri_model}, \ref{asm:dual_modelMM} hold. Let $\{x^k\}$ be generated by BMM (\eqref{eq:BMM-x} and \eqref{eq:BMM-D}). If 
    \begin{equation}\label{eq:step_BMM}
        \rho>0, \quad c_p > \beta, \quad c_d > \frac{1}{\rho}
    \end{equation}
    then for some $g_F^{k} \in \partial F(x^{k})$, it holds that $\lim_{k \to +\infty} P_{N} g_F^k = \mathbf{0}$, $\lim_{k \to +\infty}Ax^k=b$ and
    \begin{align}
        &\frac{1}{k}\sum_{t=1}^k\|P_{N}g_F^t\|^2 = O\left(\frac{1}{k}\right),~~ \min_{t\leq k}\|P_{N} g_F^{t}\|^2 = o\left(\frac{1}{k}\right)\label{eq:BMM_result_sublinear_P}\\
        &\frac{1}{k}\sum_{t=1}^k  \|Ax^t-b\|^2= O\left(\frac{1}{k}\right),~~ \min_{t\leq k}\|Ax^t-b\|^2= o\left(\frac{1}{k}\right).\label{eq:BMM_result_sublinear_Ax-b}
    \end{align}
If, in addition, $f$ is $\mu$-strongly convex, $h(x)\equiv0$, $c_p>2\beta^2/\mu$, and Assumption \ref{asm:dual subgradient in range a} holds, then 
\begin{equation}\label{eq:BMM_result_linear}
     \|z^\star-z^{k+1}\|_\Lambda^2\leq (1-\alpha)\|z^\star-z^{k}\|_\Lambda^2
 \end{equation}
where $\alpha = \frac{\min\left\{1-2\beta^2/(\mu c_p),1-1/(\rho c_d)\right\}}{5c_pc_d(1/\sigma_A^2+\rho/\mu)}\in(0,1)$.
\end{theorem}

\begin{proof}
See Appendix \ref{appen:thm_BMM}.
\end{proof}

Theorem \ref{thm:BMM} guarantees the convergence of BMM, but does not cover BMM-D where $c_p=0$. In the next theorem, we present our convergence results for BMM-D. To emphasize, we derive a linear convergence rate for BMM-D without requiring either the strong convexity of $f$ or row/column rank assumptions on $A$. 

\begin{theorem}\label{thm:BMM-D}
    \textbf{(Convergence of BMM-D)} Suppose that Assumptions \ref{asm:prob}, \ref{asm:dual_modelMM} hold, $\rho>0$, and $c_d \geq \frac{1}{\rho}$. Let $\{x^k\}$ be generated by BMM-D (\eqref{eq:MM_x} and \eqref{eq:BMM-D}). Then, for some $g_F^k \in \partial F(x^k)$, $P_N g_F^k = \textbf{0}$, $\lim_{k \to +\infty} Ax^k=b$ and \eqref{eq:BMM_result_sublinear_Ax-b} holds.
    If, in addition, $h(x) \equiv0$ and Assumption \ref{asm:dual subgradient in range a} holds, then
    \begin{align}
        &\|v^k-v^\star\|^2 \leq \frac{1}{(1+\alpha)^k} \|v^0-v^\star\|^2\label{eq:BMMD_v_rate}\\
        &\|\nabla f(x^{k+1})-\nabla f(x^\star)\|^2\le \frac{\beta}{\rho(1+\alpha)^{k-1}}\|v^0-v^\star\|^2\label{eq:BMMD_gra_rate}\\
        &\|Ax^{k+1}-b\|^2 \leq \frac{1}{(1+\alpha)^{k-1}} \cdot\frac{1}{\rho^2}\|v^0-v^\star\|^2\label{eq:BMMD_feas_rate}
 \end{align}
where $\alpha = \frac{1}{2c_d} \min\left(\frac{\sigma_A^2}{\beta},\frac{1}{\rho}\right)$.
\end{theorem} 

\begin{proof}
See Appendix \ref{appen:thm_BMM-D}
\end{proof}

\begin{remark}
    BMM and BMM-D generalize the linearized MM and MM, respectively, and the convergence rates in Theorems \ref{thm:BMM}--\ref{thm:BMM-D} are of the same order as the linearized MM \cite{xu2017accelerated,xu2021first} and MM \cite{xu2021iteration,liu2019nonergodic}. Moreover, Theorem \ref{thm:BMM-D} also establishes a linear rate for MM, which is novel to the best of our knowledge.
\end{remark}

\section{Numerical Experiment}\label{Section numerical experiment}
We demonstrate the performance of our methods in solving the following constrained regularized least squares problem:
\begin{equation}\label{eq:ls}
\begin{split}
    \underset{x\in\R^n}{\minimize}~~&~F(x) = \frac{1}{2N}\|Px-q\|^2+\lambda_1\|x\|_1+\lambda_2\|x\|_2^2\\
    \st ~&~Ax= b,
\end{split}
\end{equation}
where $N=200$, $n=200$, $\lambda_1=10^{-3}$, and $P\in\R^{N\times n}$, $q\in \R^{N}$, $A\in\R^{150\times 200}$, and $b\in\R^{150}$ are randomly generated data. We set $\lambda_2=0$ and $\lambda_2=1$ to simulate the settings of convex and strongly convex objective functions, respectively.

We run BMM and BMM-D to solve problem \eqref{eq:ls} with $\lambda_2 =0$ (convex objective function), and apply BDA, BDA-D, BMM, and BMM-D to solve problem \eqref{eq:ls} with $\lambda_2=1$ (strongly convex objective function). The experiment settings are as follows: We set $\rho=0.05$ in BMM and BMM-D. Both BDA and BMM adopt the primal cutting-plane model \eqref{eq:cpm}, BDA and BDA-D use the dual cutting-plane model in \eqref{eq:model_q}, and BMM and BMM-D employ the dual cutting-plane model in \eqref{eq:model_qrho}. We set $S_k=[k-m_p+1,k]$ in the primal cutting-plane model, and $S_k=[k-m_d+1,k]$ in the dual cutting-plane models, and take different values of $m_p,m_d$ in the experiments to test their effects. We also test the robustness of the algorithm with respect to the primal and dual step-size parameters $1/c_p$ and $1/c_d$. In all figures, we use $|f(x^k)-f(x^\star)|+ \|Ax^{k}-b\|^2$ as the optimality residual. For the convexity parameter $\mu$ and the smoothness parameter $\beta$ that will be used to determine the algorithmic parameters, we set $\mu = \frac{1}{2N} \lambda_{\min}(P^TP)+\lambda_2,~\beta = \frac{1}{2N} \lambda_{\max}(P^TP)+\lambda_2$, where  $\lambda_{\min}(P^TP)$ and $\lambda_{\max}(P^TP)$ represent the smallest and the largest eigenvalues of $P^TP$, respectively. We compare our methods with the primal-dual gradient method, DA, the linearized MM, and MM, which are specializations of BDA ($m_p=m_d=1$), BDA-D ($m_d=1$), BMM ($m_p=m_d=1$), and BMM-D ($m_d=1$), respectively.

Figs. \ref{fig:BMM-convex} -- \ref{fig:BMM-D} plot the convergence of the algorithms, where subfigure (a) displays the convergence with small step-size parameters, subfigure  (b) corresponds to large step-size parameters, and subfigure (c) plots the final optimality error generated by the methods with different primal and dual step-size parameters after a certain number (denoted by $\text{Iter}$) of iterations. In all figures (convex and strongly convex objective functions, four methods), we make similar observations:
\begin{enumerate}[1)]
    \item When using small step-sizes, the bundle size does not heavily affect the convergence, while methods with larger bundle size can achieve faster convergence by using larger step-sizes (e.g., for $m_p=m_d=10$, final optimality residual $\approx 10$ in Figure \ref{fig:BMM-convex} (a) and $\approx 10^{-3}$ in Figure \ref{fig:BMM-convex} (b)), demonstrating that our bundle scheme could accelerate the primal-dual gradient method, DA, the linearized MM, and MM.
    \item By including more cutting planes in the bundle, our methods could allow for larger primal and dual step-sizes, which eases parameter selection. This demonstrates the enhanced robustness with respect to the primal and dual step-sizes brought by the large bundle size.
    \item When using large step-sizes, increasing the bundle size from $1$ to $5$ can significantly accelerate the convergence, while further increasing it does not bring clear benefits.
\end{enumerate}
In Fig. \ref{fig:BMM}(a), the final accuracy for bundle sizes of $5,10$ is lower than that for the bundle size $1$, which is because we only solve the subproblems \eqref{eq:BMM-x} and \eqref{eq:BMM-D} to a medium accuracy, whereas they have closed-form solutions when the bundle size is $1$.

\begin{center}
\begin{minipage}[h]{1\textwidth}
\begin{figure}[H]
    \begin{subfigure}[t]{0.32\linewidth}
        \includegraphics[width=\linewidth]{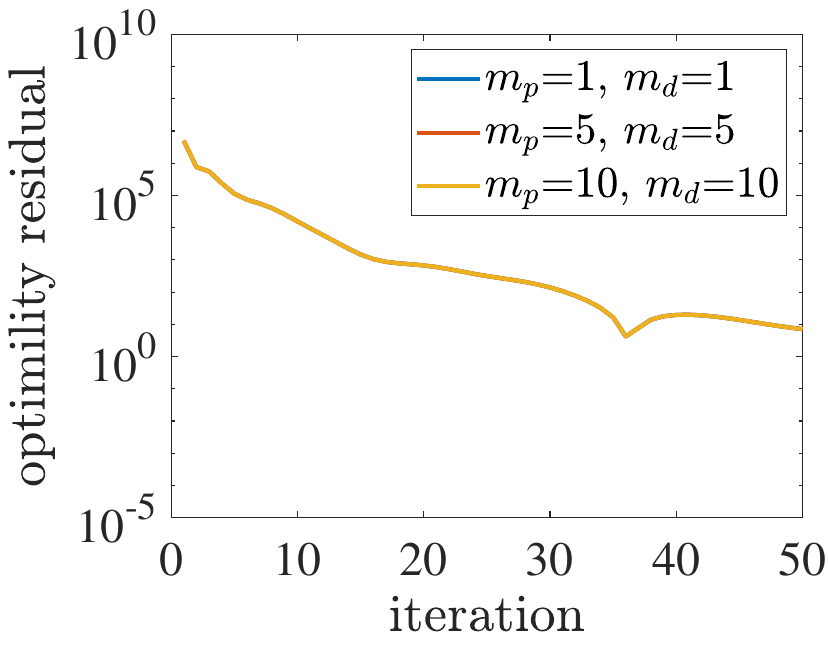}
        \caption{$\frac{1}{c_p} = \frac{1}{\beta},~\frac{1}{c_d}=\rho$}
    \end{subfigure} 
    \hfill
    \begin{subfigure}[t]{0.32\linewidth}
        \includegraphics[width=\linewidth]{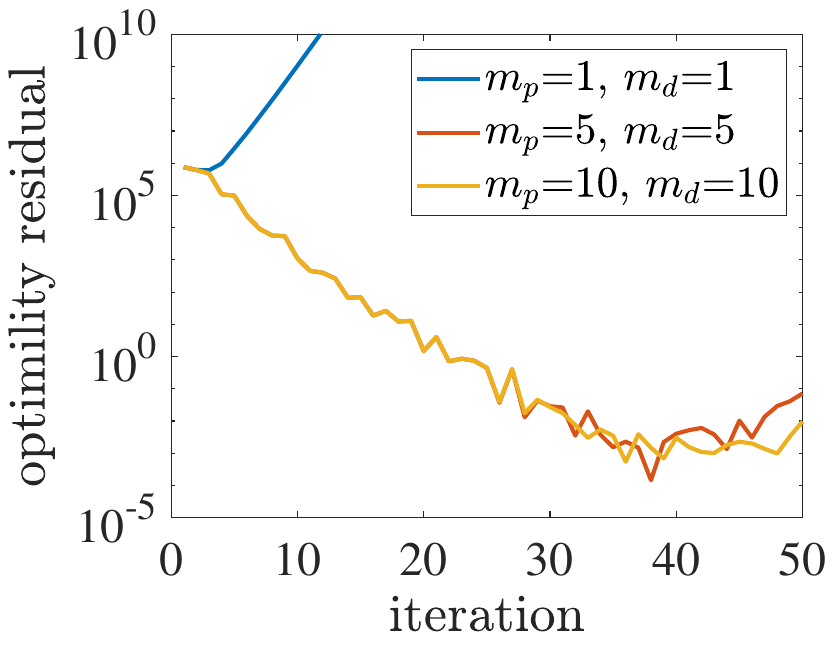}
        \caption{$\frac{1}{c_p} = \frac{4}{\beta},~\frac{1}{c_d}=2\rho$}
    \end{subfigure} 
    \hfill
    \begin{subfigure}[t]{0.32\linewidth}
        \includegraphics[width=\linewidth]{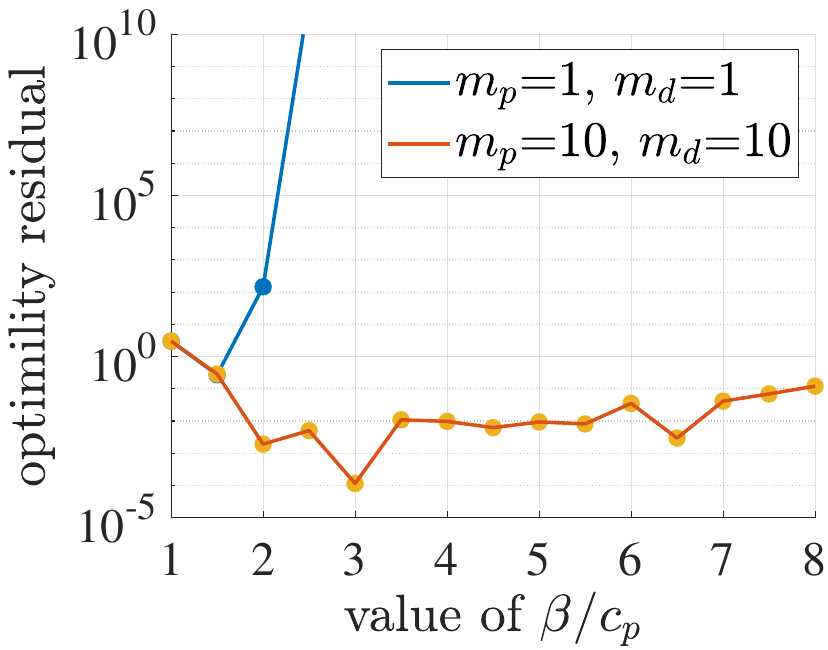}
        \caption{Iter$=50$,~$\frac{1}{c_d} = 2\rho$}
    \end{subfigure} 
    \caption{Convergence performance of BMM (convex)}
    \label{fig:BMM-convex}
\end{figure}
\end{minipage}

\begin{minipage}[h]{1\textwidth}
\begin{figure}[H]
    \centering
    \begin{subfigure}[t]{0.32\linewidth}
        \includegraphics[width=\linewidth]{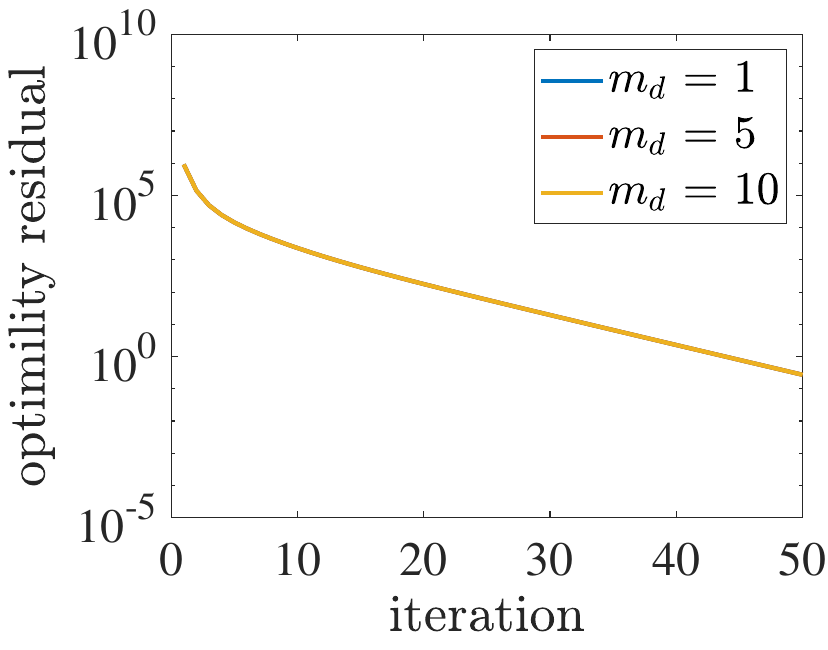}
        \caption{$\frac{1}{c_d}= \rho $}
    \end{subfigure}
    \hfill
    \begin{subfigure}[t]{0.32\linewidth}
        \includegraphics[width=\linewidth]{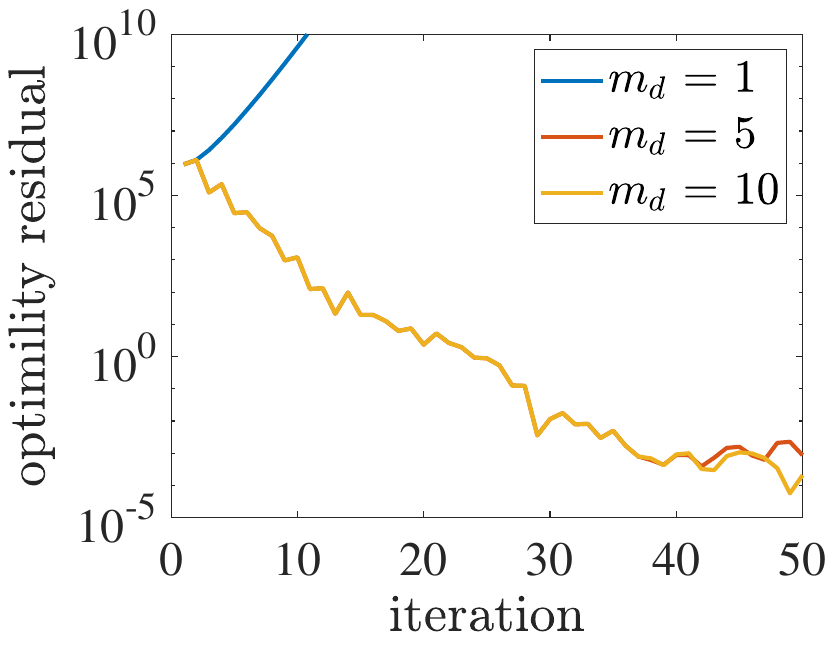}
        \caption{$\frac{1}{c_d} = 3\rho $}
    \end{subfigure}
    \hfill
    \begin{subfigure}[t]{0.32\linewidth}
        \includegraphics[width=\linewidth]{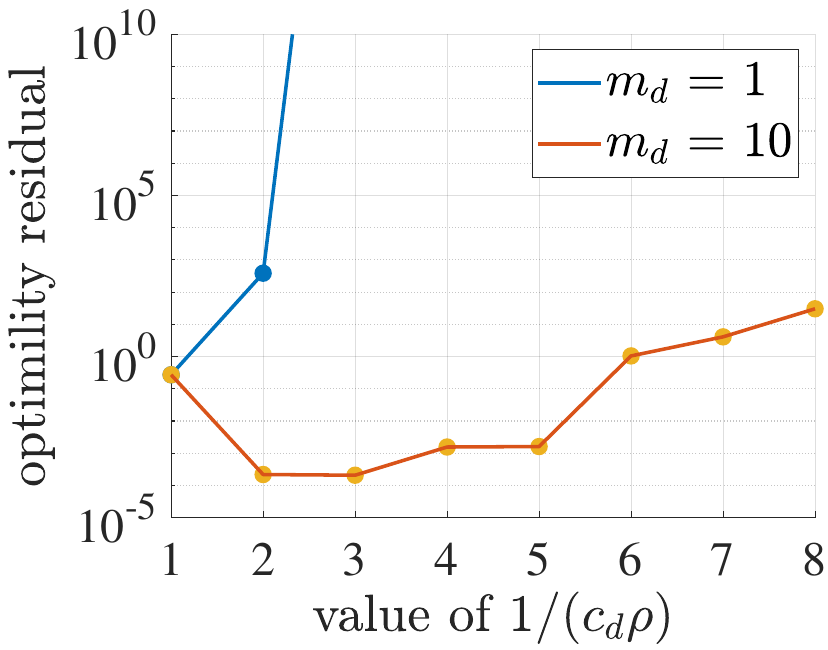}
        \caption{Iter$=50$}
    \end{subfigure}
    \caption{Convergence performance of BMM-D (convex)}
    \label{fig:BMM-D-convex}
\end{figure}
\end{minipage}

\end{center}


\begin{center}
\begin{minipage}{1\textwidth}
\begin{figure}[H]
    \centering
    \begin{subfigure}{0.32\linewidth}
        \includegraphics[width=\linewidth]{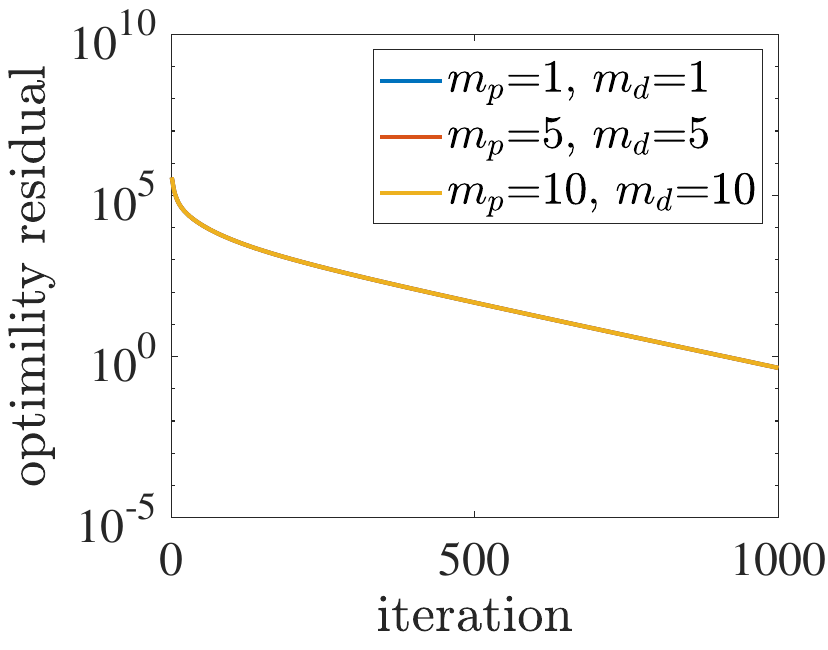}
        \caption{$\frac{1}{c_p} = \frac{1}{\beta},~\frac{1}{c_d}=\frac{\mu}{\|A\|^2}$}
    \end{subfigure}
    \hfill
    \begin{subfigure}{0.32\linewidth}
        \includegraphics[width=\linewidth]{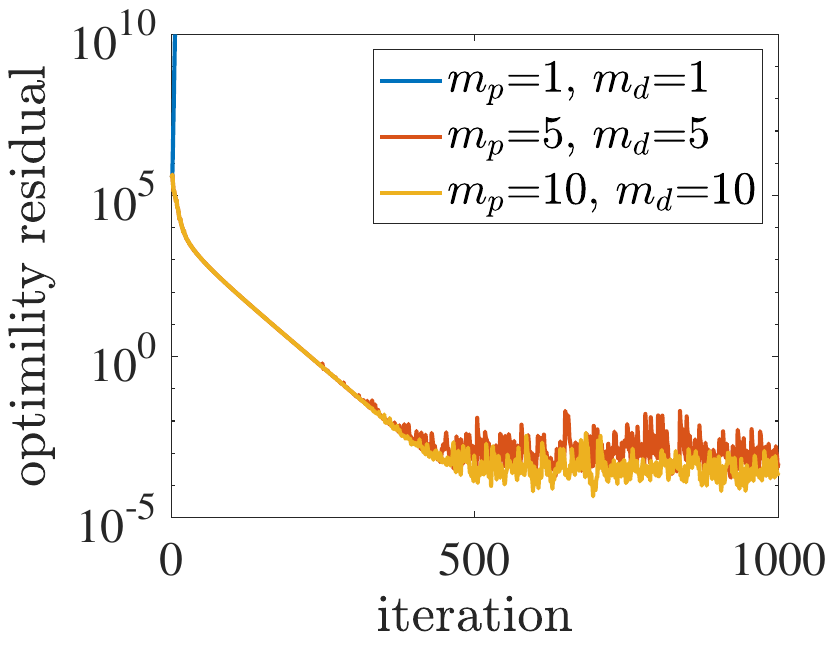}
        \caption{$\frac{1}{c_p} = \frac{3}{\beta},~\frac{1}{c_d}=\frac{4\mu}{\|A\|^2}$}
    \end{subfigure} 
    \hfill
    \begin{subfigure}{0.32\linewidth}
        \includegraphics[width=\linewidth]{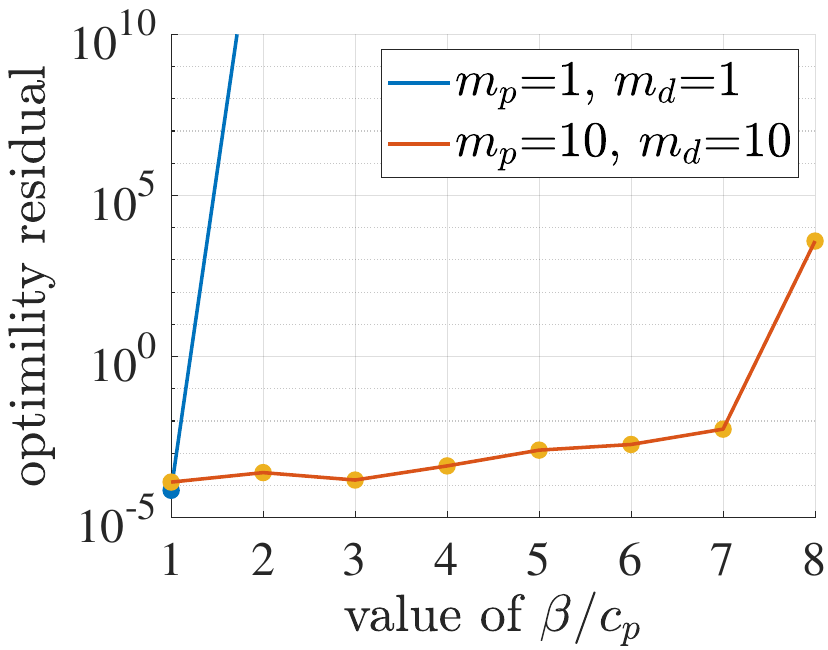}
        \caption{Iter$=1000,\frac{1}{c_d}=\frac{2\mu}{\|A\|^2}$}
    \end{subfigure} 
    \caption{Convergence performance of BDA (strongly convex)}
    \label{fig:BDA}
\end{figure}
\end{minipage}

\begin{minipage}{1\textwidth}
\begin{figure}[H]
    \centering
    \begin{subfigure}{0.32\linewidth}
        \includegraphics[width=\linewidth]{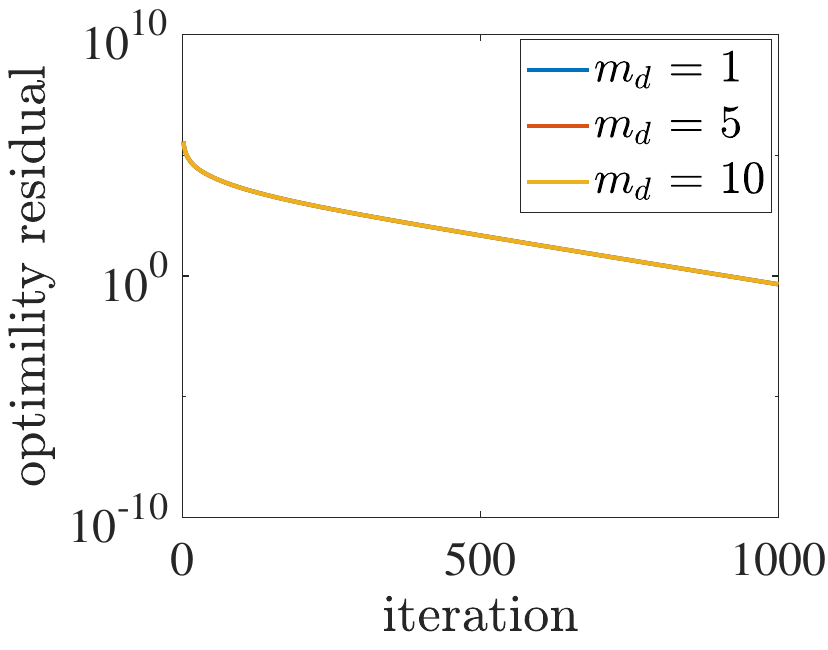}
        \caption{$\frac{1}{c_d}=\frac{\mu}{\|A\|^2}$}
    \end{subfigure}
    \hfill
    \begin{subfigure}{0.32\linewidth}
        \includegraphics[width=\linewidth]{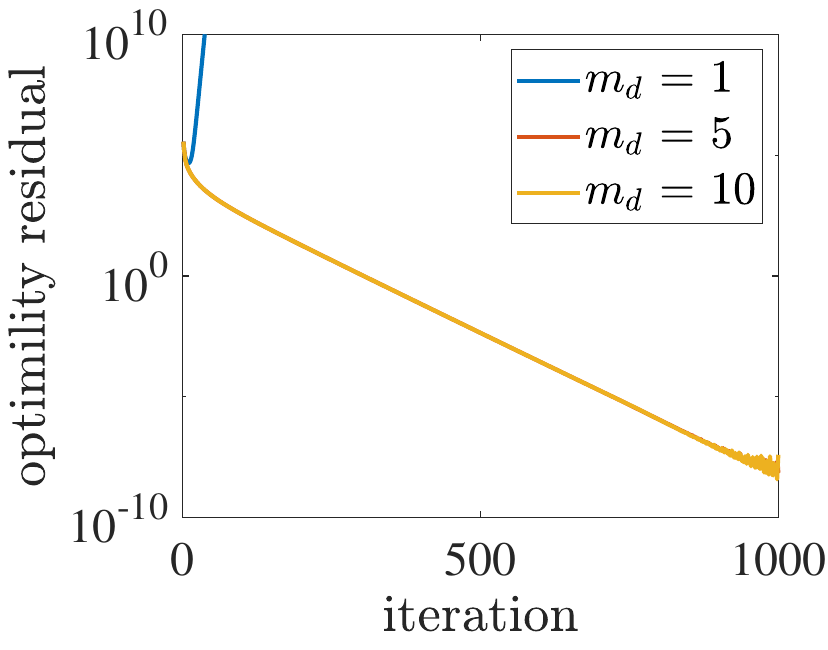}
        \caption{$\frac{1}{c_d}=\frac{3\mu}{\|A\|^2}$}
    \end{subfigure}
    \hfill
    \begin{subfigure}{0.32\linewidth}
         \includegraphics[width=\linewidth]{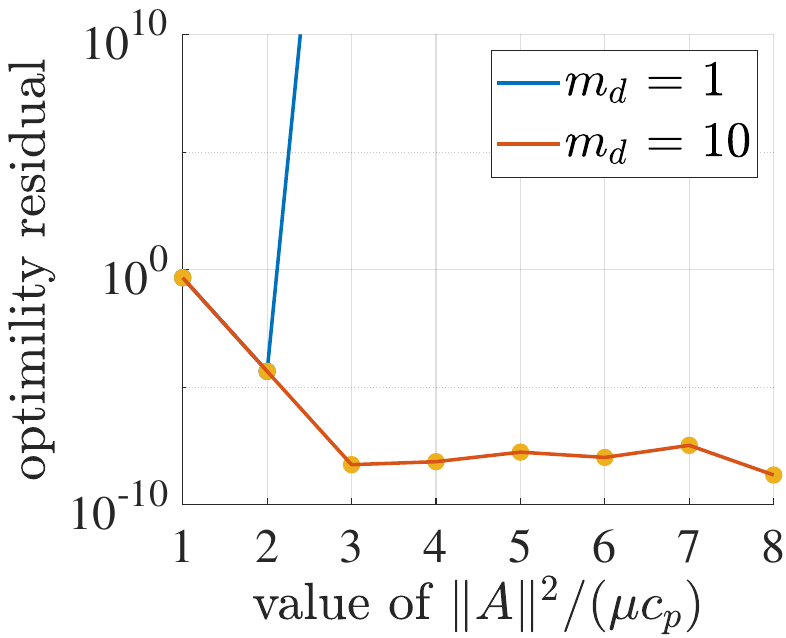}
        \caption{Iter$=1000$}
    \end{subfigure} 
    \caption{Convergence performance of BDA-D (strongly convex)}
    \label{fig:BDA-d}
\end{figure}
\end{minipage}

\begin{minipage}[h]{1\textwidth}
\begin{figure}[H]
    \centering
    \begin{subfigure}[b]{0.32\linewidth}
        \includegraphics[width=\linewidth]{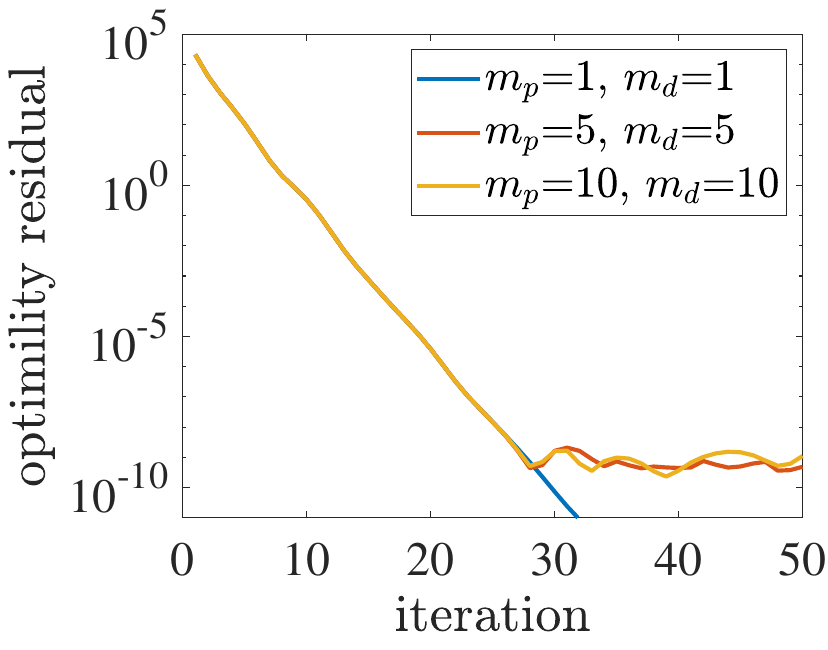}
        \caption{$\frac{1}{c_p} = \frac{1}{\beta},~\frac{1}{c_d}=\rho $}
    \end{subfigure}
    \hfill
    \begin{subfigure}[b]{0.32\linewidth}
        \includegraphics[width=\linewidth]{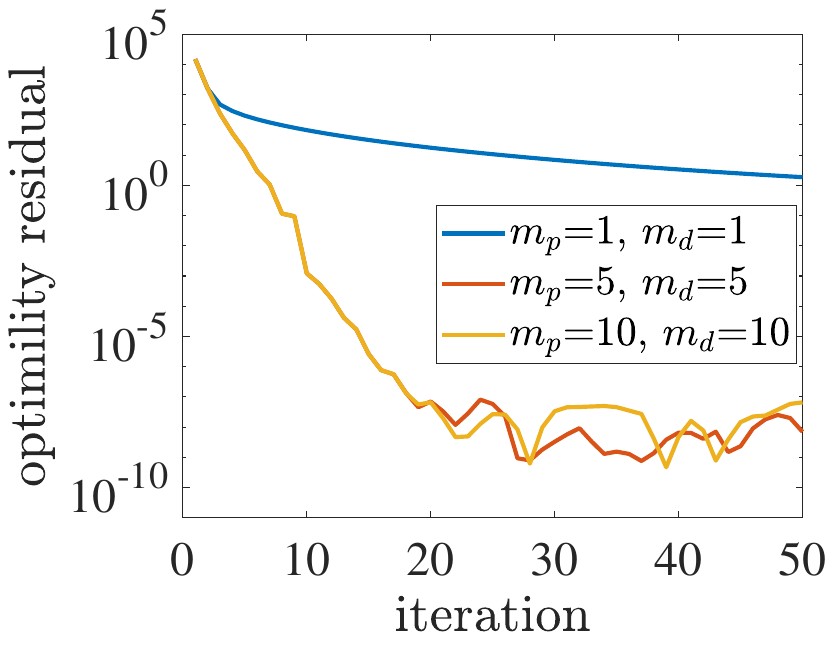}
        \caption{$\frac{1}{c_p} = \frac{2}{\beta},~\frac{1}{c_d}=2\rho $}
    \end{subfigure} 
    \hfill
    \begin{subfigure}[b]{0.32\linewidth}
        \includegraphics[width=\linewidth]{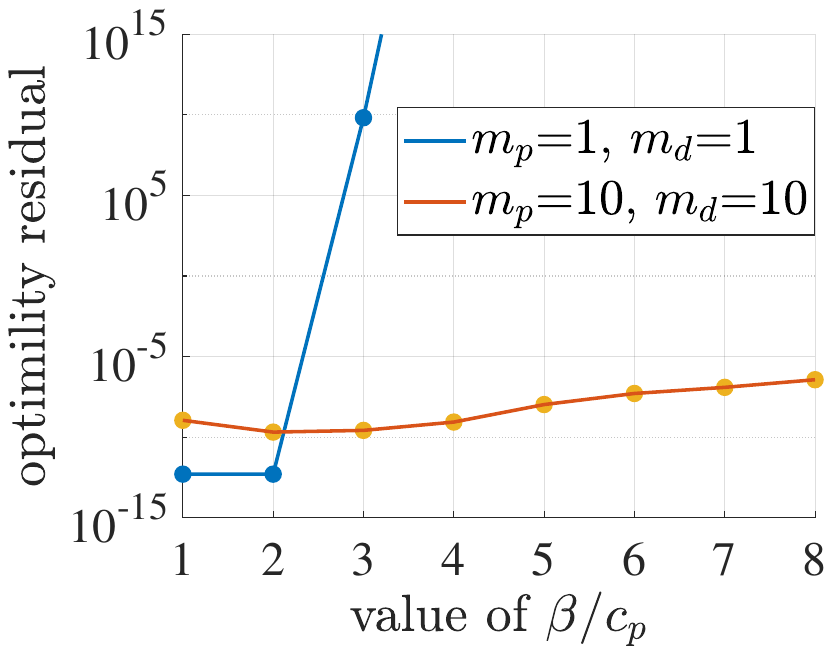}
        \caption{Iter$=50,\frac{1}{c_d} = \rho$}
    \end{subfigure} 
    \caption{Convergence performance of BMM (strongly convex)}
    \label{fig:BMM}
\end{figure}
\end{minipage}

\begin{minipage}[h]{1\textwidth}
\begin{figure}[H]
    \centering
    \begin{subfigure}[b]{0.32\linewidth}
        \includegraphics[width=\linewidth]{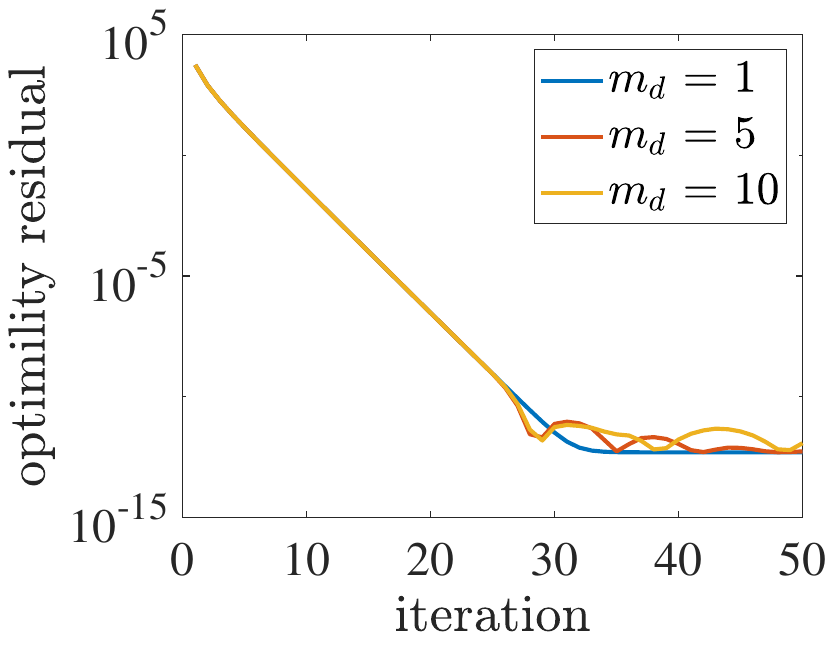}
        \caption{$\frac{1}{c_d}= \rho $}
    \end{subfigure}
    \hfill
    \begin{subfigure}[b]{0.32\linewidth}
        \includegraphics[width=\linewidth]{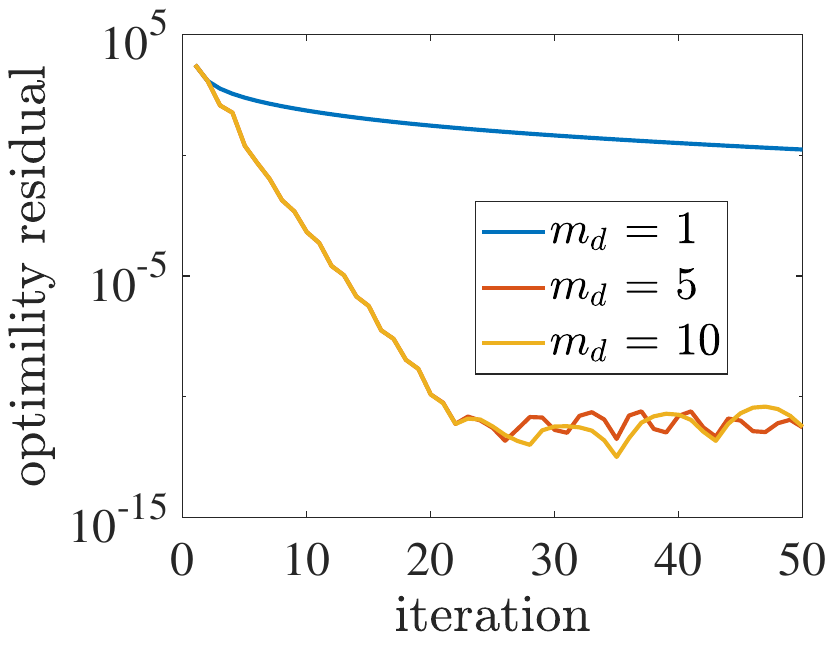}
        \caption{$\frac{1}{c_d} = 2\rho $}
    \end{subfigure}
    \hfill
    \begin{subfigure}[b]{0.32\linewidth}
        \includegraphics[width=\linewidth]{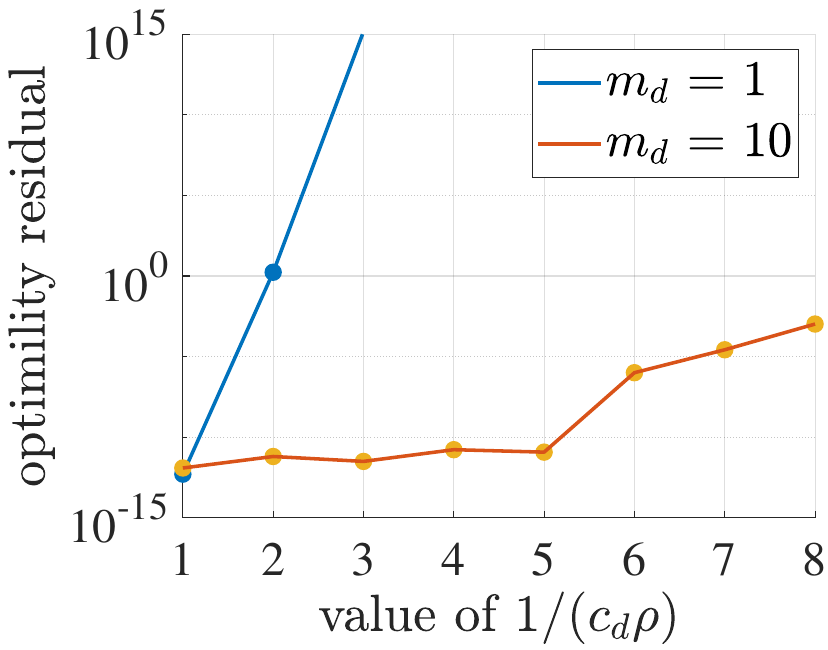}
        \caption{Iter$=50$}
    \end{subfigure}
    \caption{Convergence performance of BMM-D (strongly convex)}
    \label{fig:BMM-D}
\end{figure}
\end{minipage}

\end{center}
\section{Conclusion}\label{sec:conclusion}
We develop BDA and BMM, which adapt the general bundle model to both the primal and dual updates in DA and MM, and improve the performance of algorithms through the high approximation accuracy of the bundle model. BDA and BMM generalize DA, the primal-dual gradient method, MM, and linearized MM. Moreover, the theoretical convergence
results of BDA and BMM are no worse than and sometimes even better
than those of such existing specializations of them in terms
of rate order, problem assumption, etc. The competitive performance of the proposed methods (faster convergence and higher robustness with respect to step-size parameters) over the primal-dual gradient method, DA, the linearized MM, and MM is demonstrated by numerical experiments.

\begin{appendices}
    \section{Proof for assisting lemmas}

\subsection{Proof of Lemma \ref{lem:subproblem_dual}}\label{appen:subproblem_dual}
The Lagrange function of problem \eqref{eq:QP} is
\begin{equation*}
    L(x,t,\lambda) = t + \frac{1}{2}x^TCx + d^Tx + h(x) + \lambda^T(\tilde{A}x+\tilde{b}-t\mb{1}).
\end{equation*}
By the KKT condition,
\begin{equation}\label{eq:simplex}
    1-\mb{1}^T\lambda = \nabla_t L(x,t,\lambda) = 0, \quad \lambda\ge 0.
\end{equation}
Substituting $1-\mb{1}^T\lambda =0$ into the Lagrange function and minimizing the resulting function over $x$ yields
\begin{equation}\label{eq:QP_lagrangian}
    \begin{split}
        g(\lambda) =& \inf_x L(x,t,\lambda) = \inf_x \left\{ \frac{1}{2}x^TCx + d^Tx + h(x) + \lambda^T(\tilde{A}x+\tilde{b})\right\} \\
        =& \inf_x \left\{ h(x) +\frac{1}{2}x^TCx +d^Tx + \lambda^T\tilde{A}x \right\}+\tilde{b}^T\lambda.
    \end{split}
\end{equation}
By \eqref{eq:QP_lagrangian} and \eqref{eq:simplex}, the Lagrange dual problem of problem \eqref{eq:QP} is
\begin{equation*}
    \begin{split}
        \operatorname{maximize}  ~~~& g(\lambda)  \\
        \st ~~& 1-\mb{1}^T\lambda = \nabla_t L(x,t,\lambda) = 0, \quad \lambda\ge 0.
    \end{split}
\end{equation*}
According to the Danskin's theorem \cite{bertsekas1997nonlinear},
\[
    \nabla g(\lambda) = \tilde{A}x_{\lambda}+\tilde{b}.
\]
By the strong duality between problems \eqref{eq:QP} and \eqref{eq:subproblem_dual}, the equivalence between problem \eqref{eq:QP} and \eqref{eq:simplified_primal_prob}, and the strong convexity of the objective function in \eqref{eq:simplified_primal_prob}, we have that $x_{\lambda^\star}$ is the unique optimal solution of problem \eqref{eq:simplified_primal_prob}.

    \subsection{Proof of Lemma \ref{lemma:vstar_in_R(A)}}\label{appen:prop_vstar_in_R(A)}
By Assumption \ref{asm:prob}, there exists $v^\star$ satisfying \eqref{KKT2}. Note that both $v^\star$ and $v^0$ can be decomposed into two parts \cite[Theorem 2.26]{laub2004matrix} \[v^\star = v^\star_R + v^\star_N,\quad v^0 = v^0_R + v^0_N,\] where $v^\star_R,v^0_R \in \operatorname{Range}(A)$ and $v_N^\star,v_N^0\in\operatorname{Null}(A^T)$. Let
\[\tilde{v}^\star = v_N^0+v_R^\star.\]
It holds that
\[
\tilde{v}^\star-v^0 = v_R^\star - v_R^0 \in \operatorname{Range}(A).
\]
Moreover, it is straightforward to see $A^T\tilde{v}^\star=A^Tv^\star$, which, together with \eqref{KKT2}, yields $\mathbf{0}\in \partial F(x^\star)+A^T\tilde{v}^\star$. 

    \subsection{Proof of Lemma \ref{lemma:smooth_f_hatf}}\label{appen:lem_smooth_f_hatf}
    By the smoothness of $f$, we have \[f(x) \leq f(x^k) + \langle \nabla f (x^k),x-x^k \rangle  + \frac{\beta}{2}\|x-x^k\|^2,\] 
    which, together with Assumption \ref{asm:pri_model} (b), results in \eqref{lemma 3.1 1}.

 We prove \eqref{like smooth} by contradiction. Suppose that this fails to hold at some $\tilde{x}$, i.e., for some $\tilde{g} \in \partial \hat{f}^k(\tilde{x}) $, it holds that 
\begin{align}
    \|\tilde{g} - \nabla f(\tilde{x})\| > \beta\|\tilde{x}-x^k\| .\label{2}
\end{align}
By the convexity of $\hat{f}^k$ and \eqref{lemma 3.1 1}, for any $x$,
\begin{equation}\label{3}
\begin{split}
    \hat{f}^k(x) &\geq \hat{f}^k(\tilde{x}) + \langle \tilde{g},x-\tilde{x} \rangle\\
    & \geq f(\tilde{x}) - \frac{\beta}{2}\|\tilde{x}-x^k\|^2 + \langle \tilde{g},x-\tilde{x}\rangle\\
    &= f(\tilde{x}) + \langle \nabla f(\tilde{x}),x-\tilde{x}\rangle + \langle \tilde{g}- \nabla f(\tilde{x}),x-\tilde{x}\rangle - \frac{\beta}{2}\|\tilde{x}-x^k\|^2.
    \end{split}
\end{equation}
Let $x=\tilde{x}+\frac{1}{\beta}(\tilde{g}-\nabla f(\tilde{x}))$. Then, we have
\begin{align*}
    \langle \tilde{g}-\nabla f(\tilde{x}),x-\tilde{x}\rangle &= \frac{1}{2}\langle \tilde{g}-\nabla f(\tilde{x}),x-\tilde{x}\rangle + \frac{1}{2}\langle \tilde{g}-\nabla f(\tilde{x}),x-\tilde{x}\rangle \\
    &= \frac{1}{2\beta}\|\tilde{g}-\nabla f(\tilde{x})\|^2 + \frac{\beta}{2}\|x-\tilde{x}\|^2 \\
    &> \frac{\beta}{2}\|\tilde{x}-x^k\|^2 + \frac{\beta}{2}\|x-\tilde{x}\|^2,
\end{align*}
where the last step is due to \eqref{2}. Substituting the above equation into \eqref{3} gives 
\[
\hat{f}^k(x) > f(\tilde{x}) + \langle \nabla f(\tilde{x}),x-\tilde{x}\rangle + \frac{\beta}{2}\|x-\tilde{x}\|^2 \geq f(x),
\]
which contradicts the assumption that $\hat{f}^k(x) \leq f(x)$ for all $x$.

    \subsection{Proof of Lemma \ref{lemma:mod_sat_asm}}\label{appen:lem_mod_sat_asm}

    \subsubsection{Primal models \eqref{eq:pol_model}--\eqref{eq:two-cut}} 
    
    Since $\hat{f}^k$ in \eqref{eq:pol_model}--\eqref{eq:two-cut} are the maximum of affine functions, they are convex and satisfy Assumption \ref{asm:pri_model} (a). Assumption \ref{asm:pri_model} (b) is straightforward to see from the forms of $\hat{f}^k$ in \eqref{eq:pol_model}--\eqref{eq:two-cut}.  
    
    To show Assumption \ref{asm:pri_model} (c), note that due to the convexity of $f$,
    \[ f(x) \geq  f(x^t) + \langle \nabla f(x^t),x-x^t\rangle, \quad t \in S^k,\]
    i.e., $f(x^t) + \langle \nabla f(x^t),x-x^t\rangle, t \in S^k$ are minorants of $f$. Moreover, $\ell_f\le \min_x f(x)$. Therefore, the models \eqref{eq:pol_model}--\eqref{eq:pcpm} take the maximum of minorants of $f$, which yields $\hat{f}^k(x) \le f(x)$. For $\hat{f}^k$ in \eqref{eq:two-cut}, we show $\hat{f}^k \leq f$ by induction. By the convexity of $f$, the initial model \[\hat{f}^0(x) = f(x^0) + \langle \nabla f(x^0),x-x^0\rangle \leq f(x).\] Assume that for some $k \geq 0$, we have $\hat{f}^k(x) \leq f(x)$. By $f(x)\ge \hat{f}^k(x)$ and the convexity of $\hat{f}^k$, we have that for any $\hat{g}^{k+1}\in\partial\hat{f}^k(x^{k+1})$,
    \[f(x)\ge \hat{f}^k(x) \geq \hat{f}^k(x^{k+1}) + \langle \hat{g}^{k+1},x-x^{k+1}\rangle,\]
    which, together with the convexity of $f$, yields
    \[\hat{f}^{k+1}(x)=\max\{\hat{f}^k(x^{k+1})+\langle \hat{g}^{k+1}, x-x^k\rangle,~f(x^k)+\langle \nabla f(x^k),x-x^k\rangle\} \leq f(x).\]
    Concluding all the above, Assumption \ref{asm:pri_model} (c) holds for all $k\geq 0$.
    
    \subsubsection{Dual models \eqref{eq:model_q} and \eqref{eq:model_qrho}}
    
    Since $\hat{q}^k$ in \eqref{eq:model_q} is the minimum of affine functions, it is concave and satisfies Assumption \ref{asm:dual_modelDA} (a). For Assumption \ref{asm:dual_modelDA} (b), note that 
    \begin{equation}\nonumber
        \begin{split}
            C_q^k(v) &= L(x^{k+1},v^k)+\langle \nabla_v L(x^{k+1},v^k), v-v^k\rangle \\
            &=F(x^{k+1}) + \langle Ax^{k+1}-b,v^k\rangle + \langle Ax^{k+1}-b,v-v^k\rangle \\
            &=F(x^{k+1}) + \langle Ax^{k+1}-b,v\rangle \\
            &=L(x^{k+1},v).
        \end{split}
    \end{equation}
    By the form of $\hat{q}^k$ in \eqref{eq:model_q}, it is straightforward to see 
    $\hat{q}^k(v) \le C_q^k(v) = L(x^{k+1},v)$, i.e., Assumption \ref{asm:pri_model} (b) holds.  
    
    For Assumption \ref{asm:dual_modelDA} (c), since
    \begin{equation}\label{eq:C_q^t_majorant}
        C_q^t(v) = F(x^{t+1}) + \langle Ax^{t+1}-b,v\rangle \geq \inf_x \{ F(x) + \langle Ax-b,v\rangle \} = q(v),
    \end{equation}
    the functions $C_q^t(v),t\in S^k$ are majorants of the dual function $q$. Moreover, $u_{q,\rho}$ is a majorant of $q$. Therefore, the Polyak model, the cutting-plane model, and the Polyak cutting-plane model that take the minimum of majorants of $q$, satisfy $\hat{q}^k(v) \geq q(v)$. For the dual two-cut model, we show $\hat{q}^k(v) \geq q(v)$ by induction. The initial model \[\hat{q}^0(v) = C^0_q(v) = F(x^{1}) + \langle Ax^{1}-b,v\rangle \geq \inf_x \{ F(x) + \langle Ax-b,v\rangle\}=q(v).\]
    Assume that for some $k \ge 0$, $\hat{q}^k(v) \geq q(v)$. Then, by the concavity of $\hat{q}^k(v)$, we have that for any $-\hat{g}_q^{k+1} \in \partial (-\hat{q}^k)(v^{k+1})$,
    \[
        q(v) \leq \hat{q}^k(v) \leq \hat{q}^{k}(v^{k+1}) + \langle \hat{g}_q^{k+1}, v-v^{k+1} \rangle.
    \]
    By \eqref{eq:C_q^t_majorant}, we have $C_q^{k+1}(v) \geq q(v)$, which together with the above equation gives
    \[
    \hat{q}^{k+1}(v)  =  \min\{\hat{q}^{k}(v^{k+1})+\langle \hat{g}_q^{k+1}, v-v^{k+1}\rangle , C_q^{k+1}(v)\} \geq q(v),
    \]
    i.e., $\hat{q}^{k+1}(v)$ is a majorant of $q(v)$. Concluding all the above, Assumption \ref{asm:dual_modelDA} (c) holds for all $k\geq 0$.

    The proof of \eqref{eq:model_qrho} satisfying Assumption \ref{asm:dual_modelMM} is very similar and is therefore omitted.

\subsection{Proof of Lemma \ref{lem:dual_model_subdifferential}} \label{appen:lem_dual_model_subdifferential}
We first consider the dual models in \eqref{eq:model_q}. For the cutting-plane model, by \cite[Proposition 2.54]{Mordukhovich2023}, the subdifferential of the maximum of functions is the convex hull of the union of gradients of the ``active'' functions at $x^{k+1}$, which, together with $b=Ax^\star$, yields
\begin{align*}
    \partial (-\hat{q}^k(v))  
    =& \text{conv} \{\cup_{j\in J_k(v)}(-\nabla_v L(x^{j},v))\} \\
    =& \sum_{j \in J_k(v)} \lambda_j (b-Ax^j) \\
    =&\sum_{j \in J_k(v)} \lambda_j A(x^\star-x^j) \subseteq \operatorname{Range}(A),\nonumber
\end{align*}
where $\lambda_j \geq 0 , \sum_{j \in J_k(v)}\lambda_j = 1$ and $J_k(v) = \{j \mid \hat{q}^k(v) = F(x^j) + \langle v,Ax^j-b \rangle, j \in S^k\}$ is the ``active'' set.

For the Polyak cutting-plane model, the subdifferential can be written as
\begin{align*}
    \partial (-\hat{q}^k(v))  
    =& \text{conv} \{\cup_{j\in J_k(v)}(-\nabla_v L(x^{j},v)),\textbf{0}\}. 
\end{align*}
By $\nabla_v L(x^{j},v) = Ax^j-b = A(x^j-x^\star) \subseteq \operatorname{Range}(A)$ and $\textbf{0} \subseteq \operatorname{Range}(A)$, we have 
\begin{align*}
    \partial (-\hat{q}^k(v))  
    = \text{conv} \{\cup_{j\in J_k}(-\nabla_v L(x^{j},v^{k+1})),\textbf{0}\}  \subseteq \operatorname{Range}(A).
\end{align*}
The Polyak model is a special case of the Polyak cutting-plane model and, therefore, it satisfies \eqref{eq:partial_in_range_A}.

For the dual two-cut model, we show $\partial(-\hat{q}^k(v)) \subseteq \operatorname{Range}(A)$ by induction. The initial model 
\[
\hat{q}^0(v) = C_q^0(v) = F(x^1) + \langle Ax^1-b,v \rangle
\]
and it satisfies
\[
 \partial (-\hat{q}^0(v)) = b- Ax^1 = A(x^\star-x^1) \subseteq \operatorname{Range}(A).
\]
Assume that for some $k\geq 0$, $\partial(-\hat{q}^k(v)) \subseteq \operatorname{Range}(A)$. The two-cut model takes the form of, for any $-\hat{g}_q^{k+1}\in \partial(-\hat{q}_\rho^{k})(v^{k+1})$
\[
\hat{q}^{k+1}(v) = \min \{ \hat{q}^k(v^{k+1}) + \langle \hat{g}_q^{k+1},v-v^{k+1}\rangle, F(x^{k+1}) + \langle Ax^{k+1}-b,v \rangle \}.
\]
Therefore, the subdifferential of $-\hat{q}^{k+1}$ is 
\[
    \partial(-\hat{q}^{k+1}(v)) \subseteq \operatorname{conv}\{-\hat{g}_q^{k+1},b-Ax^{k+1}\}.
\]
By $-\hat{g}_q^{k+1} \in \partial(-\hat{q}^{k})(v^{k+1}) \subseteq \operatorname{Range}(A)$ and $b-Ax^{k+1} = A(x^\star-x^{k+1})\subseteq \operatorname{Range}(A)$, we have
\[ \partial(-\hat{q}^{k+1}(v))\subseteq \operatorname{Range}(A).\]
Concluding all the above, Assumption \ref{asm:dual subgradient in range a} holds for all $k\geq 0$.

The proof for \eqref{eq:model_qrho} is very similar to that for \eqref{eq:model_q} and is therefore omitted.

\section{Proof for BDA and BMM}
\subsection{Proof of Theorem \ref{thm:BDA}}\label{appen:thm_BDA}
Define
\begin{equation*}
    L^k(v) = -\hat{q}^k(v)+ \frac{c_d}{2} \| v - v^k\|^2.
\end{equation*}
Because $L^k(v) - \frac{c_d}{2}\|v\|^2$ is convex and $v^{k+1} = \underset{v}{\operatorname{argmin}}~L^k(v)$, we have
\begin{equation}
L^k(v^{k+1}) - L^k(v^\star) \leq  - \frac{1}{2\rho}\| v^{k+1} - v^\star\|^2. \label{Lk1-Lk2 BDA} 
\end{equation}
By Assumption \ref{asm:dual_modelDA} (b),
\begin{equation}\label{Lk1 BDA}
    \begin{split}
        L^k(v^{k+1}) &= -\hat{q}^k(v^{k+1}) + \frac{c_d}{2}\|v^{k+1}-v^k\|^2 \\
    &\geq -F(x^{k+1}) - \langle v^{k+1},Ax^{k+1}-b\rangle + \frac{c_d}{2}\|v^{k+1}-v^k\|^2 
    \end{split}
\end{equation}
and by Assumption \ref{asm:dual_modelDA} (c) and the strong duality between problem \eqref{problem} and its dual problem, 
\begin{equation}\label{Lk2 BDA}
    \begin{split}
        L^k(v^\star) &= -\hat{q}^k(v^\star) + \frac{c_d}{2}\|v^\star-v^k\|^2\\
    &\leq -q(v^\star) + \frac{c_d}{2}\|v^\star-v^k\|^2 \\
    &= -F(x^\star) + \frac{c_d}{2}\|v^\star-v^k\|^2.
    \end{split}
\end{equation}
Substituting \eqref{Lk1 BDA} and \eqref{Lk2 BDA} into \eqref{Lk1-Lk2 BDA} gives 
\begin{equation}
\begin{split}
    F(x^\star) - F(x^{k+1}) \leq& \frac{c_d}{2}\|v^\star-v^k\|^2 - \frac{c_d}{2}\|v^\star-v^{k+1}\|^2 \\
    &-\frac{c_d}{2}\|v^{k+1}-v^k\|^2 + \langle v^{k+1},Ax^{k+1}-b\rangle. \label{A BDA}
\end{split}
\end{equation}
By the first-order optimality condition of the primal update \eqref{eq:BDA-x}, we have that for some $\hat{g}^{k+1}\in \partial \hat{f}^k(x^{k+1})$ and $s^{k+1}\in \partial h(x^{k+1})$,
\begin{align*}
    \hat{g}^{k+1} +s^{k+1}+ c_p(x^{k+1}-x^k) + A^Tv^k = \mathbf{0},
\end{align*}
which, together with $b=Ax^\star$, leads to
\begin{align*}
    \langle v^{k+1},Ax^{k+1}-b\rangle &= \langle v^{k+1}-v^k,Ax^{k+1}-b\rangle + \langle v^k,Ax^{k+1}-b\rangle \\
    &=\langle v^{k+1}-v^k,Ax^{k+1}-b\rangle + \langle A^Tv^k,x^{k+1}-x^\star\rangle \\
    &=\langle v^{k+1}-v^k,Ax^{k+1}-b\rangle - \langle \hat{g}^{k+1}+s^{k+1}+c_p(x^{k+1}-x^k),x^{k+1}-x^\star\rangle.
\end{align*}
Substituting above equation into \eqref{A BDA} and using $c_p\langle x^k-x^{k+1},x^{k+1}-x^\star\rangle = \frac{c_p}{2}\|x^k-x^\star\|^2 -\frac{c_p}{2}\|x^{k+1}-x^\star\|^2-\frac{c_p}{2}\|x^{k+1}-x^k\|^2$, we can get
\begin{equation}\label{C BDA}
\begin{split}
    &F(x^\star) - F(x^{k+1}) - \langle s^{k+1},x^\star-x^{k+1}\rangle - \langle \hat{g}^{k+1},x^\star-x^{k+1}\rangle\\
    \le&\|z^\star-z^k\|^2_\Lambda -\|z^\star-z^{k+1}\|^2_\Lambda-\|z^{k+1}-z^k\|^2_\Lambda + \langle v^{k+1}-v^k,Ax^{k+1}-b\rangle\\
    =& \|z^\star-z^k\|^2_\Lambda -\|z^\star-z^{k+1}\|^2_\Lambda-\frac{c_p}{2}\|x^{k+1}-x^k\|^2- \frac{c_d}{2}\|v^{k+1}-v^k\|^2\\
    &+ \langle v^{k+1}-v^k,Ax^{k+1}-b\rangle\\
    \le&\|z^\star-z^k\|^2_\Lambda -\|z^\star-z^{k+1}\|^2_\Lambda-\frac{c_p}{2}\|x^{k+1}-x^k\|^2 + \frac{\|A\|^2}{2c_d}\|x^{k+1}-x^\star\|^2
    \end{split}
\end{equation}
where the last inequality comes from
\begin{align}
    &\langle v^{k+1}-v^k,Ax^{k+1}-b\rangle \leq \frac{c_d}{2}\|v^{k+1}-v^k\|^2+\frac{1}{2c_d}\|Ax^{k+1}-b\|^2,\label{eq:AMGM_vA}\\
    &\|Ax^{k+1}-b\|^2=\|A(x^{k+1}-x^\star)\|^2 \leq \|A\|^2\cdot\|x^{k+1}-x^\star\|^2.\nonumber
\end{align}
Due to the convexity of $h$ and $F(x) = f(x)+h(x)$, the left-hand side in \eqref{C BDA} satisfies
\begin{equation}\label{eq:bregman_fh}
\begin{split}
        & F(x^\star)-F(x^{k+1}) - \langle \hat{g}^{k+1},x^\star-x^{k+1}\rangle - \langle s^{k+1},x^\star-x^{k+1}\rangle\\
    \ge & f(x^\star)-f(x^{k+1}) - \langle \hat{g}^{k+1},x^\star-x^{k+1}\rangle\\
    = & f(x^\star)-f(x^{k+1}) - \langle \nabla f(x^{k+1}),x^\star-x^{k+1}\rangle + \langle \nabla f(x^{k+1})-\hat{g}^{k+1},x^\star-x^{k+1}\rangle.
\end{split}
\end{equation}
Using the AM-GM inequality and by \eqref{like smooth}, we have
\begin{equation*}
\begin{split}
    \langle \nabla f(x^{k+1})-\hat{g}^{k+1},x^\star-x^{k+1}\rangle& \ge - \frac{1}{\mu}\|\nabla f(x^{k+1})-\hat{g}^{k+1}\|^2 - \frac{\mu}{4}\|x^\star-x^{k+1}\|^2 \\
    &\ge - \frac{\beta^2}{\mu}\|x^{k+1}-x^k\|^2-\frac{\mu}{4}\|x^\star-x^{k+1}\|^2 ,
\end{split}
\end{equation*}
which, together with the strong convexity of $f$, yields
\begin{equation}\label{B BDA}
\begin{split}
    & f(x^\star)-f(x^{k+1}) - \langle \nabla f(x^{k+1}),x^\star-x^{k+1}\rangle+ \langle \nabla f(x^{k+1})-\hat{g}^{k+1},x^\star-x^{k+1}\rangle\\
    \geq & \frac{\mu}{2} \|x^\star-x^{k+1}\|^2  - \frac{\beta^2}{\mu}\|x^{k+1}-x^k\|^2-\frac{\mu}{4}\|x^\star-x^{k+1}\|^2 \\
    = & \frac{\mu}{4}\|x^\star-x^{k+1}\|^2 - \frac{\beta^2}{\mu}\|x^{k+1}-x^k\|^2.
    \end{split}
\end{equation}
Substituting \eqref{B BDA} into \eqref{eq:bregman_fh} and combining the resulting equation with \eqref{C BDA} gives
\begin{equation}\label{eq:xkxstar_bound}
\begin{split}
    &\left(\frac{\mu}{4}\!-\frac{\|A\|^2}{2c_d}\right)\|x^\star-x^{k+1}\|^2\\
    \le & \|z^\star-z^k\|^2_\Lambda \!-\!\|z^\star-z^{k+1}\|^2_\Lambda \!- \!\left(\frac{c_p}{2}\!-\!\frac{\beta^2}{\mu}\!\right)\|x^{k+1}-x^k\|^2.
\end{split}
\end{equation}
Using telescoping cancellation on the above equation yields that for all $K\ge 0$,
\begin{equation*}\label{summable BDA-PD}
    \left(\frac{\mu}{4}-\frac{\|A\|^2}{2c_d}\right)\sum_{k=0}^K\|x^\star-x^{k+1}\|^2 \le \|z^\star-z^0\|_{\Lambda}^2,
\end{equation*}
which implies $\lim_{k\rightarrow+\infty} \|x^k-x^\star\|=0$ and \eqref{eq:BDA_rate} according to \cite[Proposition 3.4]{shi2015extra}.

Next, we derive the linear convergence rate \eqref{eq:linear_rate_BDA} under the setting of $h(x) \equiv0$. The main idea is to show
\begin{equation}\label{eq:key_BDA_linear}
    \left(\frac{\mu}{4}\!-\frac{\|A\|^2}{2c_d}\right)\|x^\star-x^{k+1}\|^2+\left(\frac{c_p}{2}\!-\!\frac{\beta^2}{\mu}\!\right)\|x^{k+1}-x^k\|^2\ge \alpha\|z^k-z^\star\|_{\Lambda}^2,
\end{equation}
which, together with \eqref{eq:xkxstar_bound}, will give \eqref{eq:BDA_rate}.

From the optimality condition of the primal update \eqref{eq:BDA-x} and KKT condition of \eqref{problem}, we have that for some $\hat{g}^{k+1}\in \partial \hat{f}^k(x^{k+1})$,
\begin{equation}\label{eq:KKT-BDA}
    \hat{g}^{k+1} + c_p(x^{k+1}-x^k) + A^Tv^k = \mathbf{0},\qquad \nabla f(x^\star) + A^Tv^\star = \mathbf{0},
\end{equation}
which, together with the smoothness of $f$ and \eqref{like smooth}, leads to
\begin{equation}\label{eq:AT_vk_vstar}
\begin{split}
    &\|A^T(v^k-v^\star)\|^2\\
    =& \|\nabla f(x^\star)-\hat{g}^{k+1}-c_p(x^{k+1}-x^k)\|^2 \\
    \leq & 2\|\nabla f(x^\star)-\hat{g}^{k+1}\|^2 + 2 c_p^2\|x^{k+1}-x^k\|^2 \\
    \leq& 4\|\nabla f(x^\star)-\nabla f(x^{k+1})\|^2 + 4\|\nabla f(x^{k+1})-\hat{g}^{k+1}\|^2 + 2c_p^2\|x^{k+1}-x^k\|^2 \\
    \overset{\eqref{like smooth}}{\leq} &4\beta^2\|x^\star-x^{k+1}\|^2 + (4\beta^2+2c_p^2)\|x^{k+1}-x^k\|^2.
\end{split}
\end{equation}
By the first-order optimality condition on the dual update \eqref{eq:BDA-D} and \eqref{eq:partial_in_range_A},
\[c_d(v^{k+1}-v^k)\in \partial \hat{q}^k(v^{k+1})\subseteq \operatorname{Range}(A).\]
Moreover, by Lemma \ref{lemma:vstar_in_R(A)}, $v^\star-v^0\in \operatorname{Range}(A)$. Therefore,
\begin{equation}\label{eq:vk_vstar_range}
    v^k - v^\star = v^0-v^\star+\sum_{t=0}^{k-1} (v^{t+1}-v^t) \in \operatorname{Range}(A).
\end{equation}
By \eqref{eq:AT_vk_vstar} and \eqref{eq:vk_vstar_range}, we have
\begin{equation}\label{vstar-vk BDA}
\begin{split}
\|v^\star-v^k\|^2 &\leq \frac{1}{\sigma^2_A}\|A^T(v^k-v^\star)\|^2\\
&\le\frac{1}{\sigma^2_A} (4\beta^2\|x^\star-x^{k+1}\|^2 + (4\beta^2+2c_p^2)\|x^{k+1}-x^k\|^2), 
\end{split}
\end{equation}
where $\sigma_A$ is the smallest non-zero singular value of $A$ and the first inequality uses $\sigma_A^2\|w\|^2 \leq \|A^Tw\|^2$ for all $w\in \operatorname{Range}(A)$ \cite[Lemma 1]{alghunaim2020linear}.

By \eqref{vstar-vk BDA} and
\[
\begin{split}
    \frac{1}{2}\|x^\star-x^k\|^2 = \frac{1}{2}\|(x^\star-x^{k+1})+(x^{k+1}-x^k)\|^2 \le \|x^{k+1}-x^k\|^2 + \|x^\star-x^{k+1}\|^2,
\end{split}
\]
we can get
\begin{equation}\label{linear BDA-PD}
    \begin{split}
    &\frac{c_p}{2}\|x^\star-x^k\|^2 + \frac{c_d}{2}\|v^\star-v^k\|^2 \\
    \leq& (\frac{2\beta^2c_d}{\sigma^2_A}+c_p)\|x^\star-x^{k+1}\|^2+(\frac{c_d(2\beta^2+c_p^2)}{\sigma^2_A} +c_p)\|x^{k+1}-x^k\|^2.
    \end{split}
\end{equation}
Letting 
\[
\alpha' = \min \left( \frac{\frac{\mu}{4}-\frac{\|A\|^2}{2c_d}}{\frac{2\beta^2c_d}{\sigma^2_A}+c_p},\frac{\frac{c_p}{2}-\frac{\beta^2}{\mu}}{\frac{c_d(2\beta^2+c_p^2)}{\sigma^2_A} +c_p}  \right),
\]
it follows from \eqref{linear BDA-PD} that
\begin{equation}
    \begin{split}\label{eq:gamma}
    &(\frac{\mu}{4}-\frac{\|A\|^2}{2c_d})\|x^\star-x^{k+1}\|^2+(\frac{c_p}{2}-\frac{\beta^2}{\mu})\|x^{k+1}-x^k\|^2 \\
    \geq& \alpha'((\frac{2\beta^2c_d}{\sigma^2_A}+c_p)\|x^\star-x^{k+1}\|^2+(\frac{c_d(2\beta^2+c_p^2)}{\sigma^2_A} +c_p)\|x^{k+1}-x^k\|^2)\\
    \geq& \alpha' ( \frac{c_p}{2}\|x^\star-x^{k}\|^2 + \frac{c_d}{2}\|v^\star-v^k\|^2) \\
    =& \alpha' \|z^\star-z^k\|_\Lambda^2,
    \end{split}
\end{equation}
Moreover, $\alpha'$ can be expressed as
\[\alpha' = \frac{\sigma^2_A}{c_dc_p}\min\left(\frac{1/2-\|A\|^2/(\mu c_d)}{4\beta^2/(\mu c_p)+2\sigma^2_A/(\mu c_d)}, \frac{1/2-\beta^2/(\mu c_p)}{2\beta^2/c_p^2+1+\sigma^2_A/(c_pc_d)}\right),\]
and
\[
\begin{split}
& 4\beta^2/(\mu c_p)+2\sigma^2_A/(\mu c_d)\le 2+\frac{\sigma^2_A}{\|A\|^2}\le 3,\\
&2\beta^2/c_p^2+1+\sigma^2_A/(c_pc_d)\le 1/2+1+\sigma^2_A\mu^2/(4\beta^2\|A\|^2)\le 3.
\end{split}
\]
Therefore,
\[\alpha' \ge \frac{\sigma^2_A}{6c_dc_p}\min\left(1-2\beta^2/(\mu c_p), 1-2\|A\|^2/(\mu c_d)\right) = \alpha,\]
which, together with \eqref{eq:gamma}, results in \eqref{eq:key_BDA_linear}.

\subsection{Proof of Theorem \ref{thm:BDA-D}}\label{appen:thm_BDA_D}
Recall that BDA-D is a special case of BDA where $\hat{f}^k=f$ and $c_p=0$, so that the derivations in Appendix \ref{appen:thm_BDA} still hold. However, since $\hat{f}^k=f$, we have $\nabla f(x^{k+1})=\hat{g}^{k+1}$, so that \eqref{B BDA} reduces to
\begin{equation}\label{eq:bregman_f}
\begin{split}
    & f(x^\star)-f(x^{k+1}) - \langle \nabla f(x^{k+1}),x^\star-x^{k+1}\rangle+ \langle \nabla f(x^{k+1})-\hat{g}^{k+1},x^\star-x^{k+1}\rangle\\
    =& f(x^\star)-f(x^{k+1}) - \langle \nabla f(x^{k+1}),x^\star-x^{k+1}\rangle\\
    \ge& \frac{\mu}{2}\|x^\star-x^{k+1}\|^2.
    \end{split}
\end{equation}
The last step is due to the strong convexity of $f$. Substituting \eqref{eq:bregman_f} into \eqref{eq:bregman_fh} and combining the resulting equation with \eqref{C BDA} and $c_p=0$ gives
\begin{equation}\label{summable BDA-D}
\begin{split}
    \left(\frac{\mu}{2}-\frac{\|A\|^2}{2c_d}\right)\|x^\star-x^{k+1}\|^2 \le & \frac{c_d}{2}(\|v^\star-v^k\|^2-\|v^\star-v^{k+1}\|^2).
\end{split}
\end{equation}
Using telescoping cancellation on the above equation yields
\begin{equation*}
    \left(\frac{\mu}{2}-\frac{\|A\|^2}{2c_d}\right)\sum_{k=0}^K\|x^\star-x^{k+1}\|^2 \le \frac{c_d}{2}\|v^\star-v^0\|^2,
\end{equation*}
which implies $\lim_{k\rightarrow+\infty} \|x^k-x^\star\|=0$ and \eqref{eq:BDA_rate} according to \cite[Proposition 3.4]{shi2015extra}.

Next, we derive the linear convergence. Since $\hat{f}^k=f$ and $c_p=0$, equation \eqref{eq:KKT-BDA} reduces to
\begin{equation*}\label{eq:KKT-BDAD}
    \nabla f(x^{k+1}) + A^Tv^k = \mathbf{0},\qquad \nabla f(x^\star) + A^Tv^\star = \mathbf{0},
\end{equation*}
which, together with the smoothness of $f$, leads to
\begin{equation}\label{eq:AT_vk_vstar_BDAD}
\begin{split}
    \|A^T(v^k-v^\star)\|^2=& \|\nabla f(x^\star)-\nabla f(x^{k+1})\|^2\\
    \le &\beta^2\|x^\star-x^{k+1}\|^2.
\end{split}
\end{equation}
By \eqref{eq:AT_vk_vstar_BDAD} and the first step in \eqref{vstar-vk BDA}, we have
\begin{equation*}
\begin{split}
\|v^\star-v^k\|^2 &\leq \frac{1}{\sigma^2_A}\|A^T(v^k-v^\star)\|^2\\
&\le\frac{\beta^2}{\sigma^2_A}\|x^\star-x^{k+1}\|^2. 
\end{split}
\end{equation*}
Substituting the above equation into \eqref{summable BDA-D} gives
\begin{align*}
    \|v^\star-v^{k+1}\|^2 \leq (1-\alpha)\|v^\star-v^{k}\|^2 \le (1-\alpha)^{k+1}\|v^\star-v^0\|^2,
\end{align*}
which, together with \eqref{summable BDA-D}, results in
\begin{align*}
    \frac{1}{2}(\mu-\|A\|^2/c_d)\|x^\star-x^{k+1}\|^2\le\frac{c_d}{2}\|v^\star-v^k\|^2 \leq  \frac{c_d}{2}(1-\alpha)^k \cdot\|v^\star-v^{0}\|^2.
\end{align*}

\subsection{Proof of Theorem \ref{thm:BMM}}\label{appen:thm_BMM}
Define 
\begin{equation*}
    L_\rho^k(v) = -\hat{q}^k_\rho(v)+ \frac{c_d}{2} \| v - v^k\|^2.
\end{equation*}
Because $L^k_\rho(v) - \frac{c_d}{2}\|v\|^2$ is convex and $v^{k+1} = \arg\min_v L^k_\rho(v)$, we have 
\begin{equation}\label{eq:Lk1-Lk2_BMM} 
L^k_\rho(v^{k+1}) - L^k_\rho(v^\star) \leq  - \frac{c_d}{2}\| v^{k+1} - v^\star\|^2. 
\end{equation}
By Assumption \ref{asm:dual_modelMM} (b),
\begin{equation}\label{eq:Lk1_BMM}
    \begin{split}
        L^k_\rho(v^{k+1}) =& -\hat{q}^k(v^{k+1}) + \frac{c_d}{2}\|v^{k+1}-v^k\|^2 \\
    \geq& -F(x^{k+1}) - \langle v^{k+1},Ax^{k+1}-b\rangle -\frac{\rho}{2}\|Ax^{k+1}-b\|^2 \\ &+ \frac{c_d}{2}\|v^{k+1}-v^k\|^2,  
    \end{split}
\end{equation}
and by Assumption \ref{asm:dual_modelMM} (c) and the strong duality between problem \eqref{problem} and its dual problem,
\begin{equation}\label{eq:Lk2_BMM}
    \begin{split}
        L^k_\rho(v^\star) &= -\hat{q}^k_\rho(v^\star) + \frac{c_d}{2}\|v^\star-v^k\|^2 \\
    &\leq -q_\rho(v^\star) + \frac{c_d}{2}\|v^\star-v^k\|^2 \\
    &= -F(x^\star) + \frac{c_d}{2}\|v^\star-v^k\|^2.
    \end{split}
\end{equation}
Substituting \eqref{eq:Lk1_BMM} and \eqref{eq:Lk2_BMM} into \eqref{eq:Lk1-Lk2_BMM} gives 
\begin{equation}\label{eq:A_BMM}
    \begin{split}
        &F(x^\star) - F(x^{k+1}) -\frac{\rho}{2}\|Ax^{k+1}-b\|^2 \\
    \leq& \frac{c_d}{2}\|v^\star-v^k\|^2 - \frac{c_d}{2}\|v^\star-v^{k+1}\|^2 -\frac{c_d}{2}\|v^{k+1}-v^k\|^2 + \langle v^{k+1},Ax^{k+1}-b\rangle \\
    \leq& \frac{c_d}{2}\|v^\star-v^k\|^2 - \frac{c_d}{2}\|v^\star-v^{k+1}\|^2 -\frac{c_d}{2}\|v^{k+1}-v^k\|^2 + \langle v^{k},Ax^{k+1}-b\rangle \\
    &+ \langle v^{k+1}-v^k, Ax^{k+1}-b\rangle.
    \end{split}
\end{equation}
By the first-order optimality condition of the primal update \eqref{eq:BMM-x}, we have that for some $\hat{g}^{k+1} \in \partial\hat{f}^k(x^{k+1})$ and $s^{k+1}\in \partial h(x^{k+1})$,
\begin{equation}\label{eq:BMM_first_order}
    \hat{g}^{k+1} + s^{k+1}+c_p(x^{k+1}-x^k) + A^Tv^k +\rho A^T(Ax^{k+1}-b)= \mathbf{0},
\end{equation}
which, together with $b=Ax^\star$, leads to
\begin{equation}\label{eq:vkAxk1+b}
    \begin{split}
    &\langle v^k,Ax^{k+1}-b\rangle\\
    =&\langle A^Tv^k,x^{k+1}-x^\star\rangle \\
    =&- \langle \hat{g}^{k+1}+s^{k+1}+c_p(x^{k+1}-x^k)+\rho A^T(Ax^{k+1}-b),x^{k+1}-x^\star\rangle\\
    =&-\langle \hat{g}^{k+1}+s^{k+1}+c_p(x^{k+1}-x^k),x^{k+1}-x^\star\rangle- \rho\|Ax^{k+1}-b\|^2\\
    =& -\langle \hat{g}^{k+1}+s^{k+1},x^{k+1}-x^\star\rangle- \rho\|Ax^{k+1}-b\|^2\\
    &+\frac{c_p}{2}\|x^k-x^\star\|^2 -\frac{c_p}{2}\|x^{k+1}-x^\star\|^2-\frac{c_p}{2}\|x^{k+1}-x^k\|^2,
    \end{split}
\end{equation}
where the last step uses
\[c_p\langle x^k-x^{k+1},x^{k+1}-x^\star\rangle = \frac{c_p}{2}\|x^k-x^\star\|^2 -\frac{c_p}{2}\|x^{k+1}-x^\star\|^2-\frac{c_p}{2}\|x^{k+1}-x^k\|^2.\]
Substituting \eqref{eq:vkAxk1+b} into \eqref{eq:A_BMM} and using \eqref{eq:AMGM_vA}, we can get
\begin{equation}\label{eq:C_BMM}
    \begin{split}
    &F(x^\star) - F(x^{k+1}) - \langle s^{k+1}+\hat{g}^{k+1},x^\star-x^{k+1}\rangle  +\frac{\rho}{2}\|Ax^{k+1}-b\|^2 \\
    \leq& \|z^\star-z^k\|^2_\Lambda -\|z^\star-z^{k+1}\|^2_\Lambda-\frac{c_p}{2}\|x^{k+1}-x^k\|^2- \frac{c_d}{2}\|v^{k+1}-v^k\|^2\\
    &+ \langle v^{k+1}-v^k,Ax^{k+1}-b\rangle\\
    \le&\|z^\star-z^k\|^2_\Lambda -\|z^\star-z^{k+1}\|^2_\Lambda-\frac{c_p}{2}\|x^{k+1}-x^k\|^2 + \frac{1}{2c_d}\|Ax^{k+1}-b\|^2.
    \end{split}
\end{equation}
Due to  $F(x) = f(x)+h(x)$ and the convexity of $h$, we have 
\begin{equation*}\label{eq:BMM-Fxstar-Fxk1_ub}
    \begin{split}
    & F(x^\star) - F(x^{k+1}) - \langle s^{k+1}+\hat{g}^{k+1},x^\star-x^{k+1}\rangle\\
    \geq& f(x^\star)-f(x^{k+1}) - \langle \hat{g}^{k+1},x^\star-x^{k+1}\rangle\\
    \overset{\eqref{lemma 3.1 1}}{\geq}& f(x^\star) - \hat{f}^k(x^{k+1}) - \frac{\beta}{2}\|x^{k+1}-x^k\|^2 - \langle \hat{g}^{k+1},x^\star-x^{k+1} \rangle\\
    \geq& \hat{f}^k(x^\star) - \hat{f}^k(x^{k+1})  - \langle \hat{g}^{k+1},x^\star-x^{k+1} \rangle- \frac{\beta}{2}\|x^{k+1}-x^k\|^2\\
    \geq& -\frac{\beta}{2}\|x^{k+1}-x^k\|^2,
    \end{split}
\end{equation*}
where the last two steps use Assumption \ref{asm:pri_model} (c) and Assumption \ref{asm:pri_model} (a), respectively. Substituting above equation into \eqref{eq:C_BMM} gives
\begin{equation}\label{sublinear}
    \begin{split}
    &\frac{c_p-\beta}{2}\|x^{k+1}-x^k\|^2 + (\frac{\rho}{2}-\frac{1}{2c_d})\|Ax^{k+1}-b\|^2  \\
    \leq& \|z^\star-z^k\|^2_\Lambda -\|z^\star-z^{k+1}\|^2_\Lambda .  
    \end{split}
\end{equation}
Since $P_N$ is the projection matrix onto $\operatorname{Null}(A)$, $\operatorname{Range}(A^T)\perp\operatorname{Null}(A)$, and $A^Tv^k, A^T(Ax^{k+1}-b) \in \operatorname{Range}(A^T)$, we know that
\[
    P_NA^Tv^k =P_NA^T(Ax^{k+1}-b)=\mathbf{0}.
\]
Multiply both sides of \eqref{eq:BMM_first_order} by $P_N$ and using the equation above, we can get 
$P_N(\hat{g}^{k+1}+s^{k+1})+c_pP_N(x^{k+1}-x^k) = \mathbf{0}$, which implies
\begin{equation}\label{P BMM-PD}
    \begin{split}
        \|P_N(\hat{g}^{k+1}+s^{k+1})\| &= c_p\|P_N(x^{k+1}-x^k)\| \\
    &\leq c_p \|x^{k+1}-x^k\|, 
    \end{split}
\end{equation}
where the last inequality uses $\|P_N\|_2\leq 1$ ($P_N$ is a projection matrix). Letting $g_F^k=\nabla f(x^k)+s^k$ for all $k\ge 0$, it follows that $g_F^k \in \partial F(x^k)$. By \eqref{like smooth} and \eqref{P BMM-PD},
\begin{equation*}
\begin{split}
    \|P_N g_F^{k+1}\|
    &= \|P_N(\nabla f(x^{k+1})+s^{k+1})\|\\
    &= \|P_N(\nabla f(x^{k+1})-\hat{g}^{k+1})+P_N(\hat{g}^{k+1}+s^{k+1})\|\\
    &\le \|\nabla f(x^{k+1})-\hat{g}^{k+1}\| + \|P_N(\hat{g}^{k+1}+s^{k+1})\|\\
    &\leq (\beta+c_p)\|x^{k+1}-x^k\|,
\end{split}
\end{equation*}
where the second last step uses $\|P_N\|_2\le 1$. Substituting above inequality into \eqref{sublinear} gives
\begin{equation*}\label{sublinear of BMM}
\begin{split}
    &\frac{c_p-\beta}{2(c_p+\beta)^2}\|P_N g_F^{k+1}\|^2+(\frac{\rho}{2}-\frac{1}{2c_d})\|Ax^{k+1}-b\|^2\\
    \le&\|z^\star-z^k\|^2_\Lambda -\|z^\star-z^{k+1}\|^2_\Lambda .
\end{split}
\end{equation*}
Using telescoping cancellation on the above equation yields that for all $K\geq 0$,
\begin{equation}\nonumber
\begin{split}
    &\frac{c_p-\beta}{2(c_p+\beta)^2}\sum_{k=0}^K\|P_N g_F^{k+1}\|^2 \leq \|z^\star-z^0\|_\Lambda^2, \\
    &(\frac{\rho}{2}-\frac{1}{2c_d})\sum_{k=0}^K\|Ax^{k+1}-b\|^2\leq \|z^\star-z^0\|_\Lambda^2,
\end{split}
\end{equation}
which implies $\lim_{k \to +\infty}\|P_Ng_F^k\|=0$ and $\lim_{k \to +\infty}\|Ax^k-b\|=0$ and \eqref{eq:BMM_result_sublinear_P}--\eqref{eq:BMM_result_sublinear_Ax-b} according to \cite[Proposition 3.4]{shi2015extra}. 

Next we prove the linear convergence rate 
\eqref{eq:BMM_result_linear}. By the strong convexity of $f$ and $h(x) \equiv0$, the left-hand side of \eqref{eq:C_BMM} can be rewritten as
\begin{equation}\label{eq:BMM_str_case}
    \begin{split}
    &F(x^\star) - F(x^{k+1}) - \langle s^{k+1}+\hat{g}^{k+1},x^\star-x^{k+1}\rangle  +\frac{\rho}{2}\|Ax^{k+1}-b\|^2\\
    =&f(x^\star)-f(x^{k+1}) - \langle \hat{g}^{k+1},x^\star-x^{k+1}\rangle +\frac{\rho}{2}\|Ax^{k+1}-b\|^2\\
    =&f(x^\star)-f(x^{k+1}) - \langle \nabla f(x^{k+1}),x^\star-x^{k+1}\rangle \\
    &+ \langle \nabla f(x^{k+1})-\hat{g}^{k+1},x^\star-x^{k+1}\rangle +\frac{\rho}{2}\|Ax^{k+1}-b\|^2\\
    \geq& \frac{\mu}{2}\|x^\star-x^{k+1}\|^2 + \langle \nabla f(x^{k+1})-\hat{g}^{k+1},x^\star-x^{k+1}\rangle +\frac{\rho}{2}\|Ax^{k+1}-b\|^2.
    \end{split}
\end{equation}
Using the AM-GM inequality and \eqref{like smooth}, we have
\begin{equation}\nonumber
    \begin{split}
        \langle \nabla f(x^{k+1})-\hat{g}^{k+1},x^\star-x^{k+1}\rangle &\geq -\frac{1}{\mu}\|\nabla f(x^{k+1})-\hat{g}^{k+1}\|^2-\frac{\mu}{4}\|x^\star-x^{k+1}\|^2 \\
        &\geq -\frac{\beta^2}{\mu}\|x^{k+1}-x^k\|^2- \frac{\mu}{4}\|x^\star-x^{k+1}\|^2.
    \end{split}
\end{equation}
Substituting above into \eqref{eq:BMM_str_case} and combining the resulting equation with \eqref{eq:C_BMM} gives
\begin{equation}\label{summable 2 BMM-PD}
    \begin{split}
        &\frac{\mu}{4}\|x^\star-x^{k+1}\|^2 +(\frac{c_p}{2}-\frac{\beta^2}{\mu})\|x^{k+1}-x^k\|^2 + (\frac{\rho}{2}-\frac{1}{2c_d})\|Ax^{k+1}-b\|^2\\
    \leq& \|z^k-z^\star\|^2_\Lambda - \|z^{k+1}-z^\star\|^2_\Lambda .
    \end{split}
\end{equation}
From the first-order optimality condition of the primal update \eqref{eq:BMM-x} and KKT condition of \eqref{problem}, we have that for some $\hat{g}^{k+1} \in \hat{f}^k(x^{k+1})$
\begin{equation}\label{eq:smooth_KKT}
    \hat{g}^{k+1} + c_p(x^{k+1}-x^k) + A^Tv^k + \rho A^T(Ax^{k+1}-b)= \mathbf{0}, \quad \nabla f(x^\star) + A^Tv^\star = \mathbf{0}.
\end{equation}
By \eqref{eq:partial_in_range_A} and the first-order optimality condition on dual update \eqref{eq:BMM-D},
\[c_d(v^{k+1}-v^k)\in \partial \hat{q}_\rho^k(v^{k+1})\subseteq \operatorname{Range}(A),\]
which, together with $v^\star-v^0\in \operatorname{Range}(A)$, implies
\begin{equation}\label{eq:BMM_vk_vstar_range}
    v^k - v^\star = v^0-v^\star+\sum_{t=0}^{k-1} (v^{t+1}-v^t) \in \operatorname{Range}(A).
\end{equation}
Letting $(A^T)^\dag$ be the Moore-Penrose pseudo-inverse of $A^T$, according to \cite[Proposition 4.9.2, Page 318]{petersen2012linear}, $(A^T)^\dag A^T$ is the projection matrix onto $\operatorname{Range}(A)$, which, together with \eqref{eq:BMM_vk_vstar_range} and \eqref{eq:smooth_KKT}, implies
\begin{equation*}
\begin{split}
    v^k - v^\star &= (A^T)^\dag A^T(v^k-v^\star)\\
    &= (A^T)^\dag (\nabla f(x^\star)-\hat{g}^{k+1}-c_p(x^{k+1}-x^k)-\rho A^T(Ax^{k+1}-b))\\
    &= (A^T)^\dag (\nabla f(x^\star)-\hat{g}^{k+1}-c_p(x^{k+1}-x^k))-\rho (Ax^{k+1}-b).
\end{split}
\end{equation*}
By the above equation, the smoothness of $f$, and \eqref{like smooth}, we have
\begin{equation}\label{eq:BMM_Hello}
    \begin{split}
    &\|v^k-v^\star\|^2 \\
    =& \|(A^T)^\dag(\nabla f(x^\star)-\hat{g}^{k+1}-c_p(x^{k+1}-x^k))-\rho (Ax^{k+1}-b)\|^2\\
    =& \|(A^T)^\dag(\nabla f(x^\star)-\nabla f(x^{k+1})\!+\!\nabla f(x^{k+1})\! -\! \hat{g}^{k+1}\!-\!c_p(x^{k+1}-x^k))\!-\!\rho (Ax^{k+1}-b)\|^2\\
    \leq& 4\|(A^T)^\dag(\nabla f(x^\star)-\nabla f(x^{k+1}))\|^2 + 4\|(A^T)^\dag(\nabla f(x^{k+1})-\hat{g}^{k+1})\|^2 \\ &+ 4 c_p^2\|(A^T)^\dag(x^{k+1}-x^k)\|^2+ 4\rho^2\|Ax^{k+1}-b\|^2 \\
        \leq& \frac{4\beta^2}{\sigma_A^2}\|x^{k+1}-x^\star\|^2 + \frac{4(\beta^2+c_p^2)}{\sigma_A^2}\|x^{k+1}-x^k\|^2 + 4\rho^2\|Ax^{k+1}-b\|^2,
    \end{split}
\end{equation}
where the last step uses $\|(A^T)^\dag\|=\frac{1}{\sigma_A}$. Moreover, \[\frac{1}{2}\|x^k-x^\star\|^2 \leq \|x
^{k+1}-x^\star\|^2 + \|x^{k+1}-x^k\|^2,\]
which, together with \eqref{eq:BMM_Hello}, yields
\begin{equation}\label{eq:BMM_Hi}
    \begin{split}
        &\frac{c_p}{2}\|x^k-x^\star\|^2+\frac{c_d}{2}\|v^k-v^\star\|^2 \\
        \leq& (\frac{2\beta^2c_d}{\sigma_A^2}+c_p)\|x^{k+1}-x^\star\|^2 + (\frac{2\beta^2c_d+2c_p^2c_d}{ \sigma_A^2}+c_p)\|x^{k+1}-x^k\|^2 \\
    &+2\rho^2c_d\|Ax^{k+1}-b\|^2.
    \end{split}
\end{equation}
Letting 
\begin{align*}
     \alpha' = 
     \min \{ \frac{\frac{\mu}{4}}{\frac{2\beta^2c_d}{\sigma_A^2}+c_p},\frac{\frac{c_p}{2}-\frac{\beta^2}{\mu}}{\frac{2\beta^2c_d+2c_p^2c_d}{\sigma_A^2}+c_p},\frac{\rho-\frac{1}{c_d}}{4\rho^2c_d}\},  
\end{align*}
it follows from \eqref{eq:BMM_Hi} that 
\begin{equation}\label{eq:three_term_sum}
\begin{split}
        &\frac{\mu}{4}\|x^\star-x^{k+1}\|^2 + (\frac{c_p}{2}-\frac{\beta^2}{\mu})\|x^{k+1}-x^k\|^2+ (\frac{\rho}{2}-\frac{1}{2c_d})\|Ax^{k+1}-b\|^2 \\ 
    \ge &  \alpha' \biggl((\frac{2\beta^2c_d}{\sigma_A^2}+c_p)\|x^{k+1}-x^\star\|^2 + (\frac{2\beta^2c_d+2c_p^2c_d}{ \sigma_A^2}+c_p)\|x^{k+1}-x^k\|^2\\
    &+2\rho^2c_d\|Ax^{k+1}-b\|^2) \biggr)\\
    \ge & \alpha' \left(\frac{c_p}{2}\|x^k-x^\star\|^2+\frac{c_d}{2}\|v^k-v^\star\|^2\right).
\end{split}
\end{equation}
Moreover, $\alpha'$ can be expressed as
\[
\alpha' = \frac{\sigma_A^2}{c_pc_d}\min \left(\frac{1}{\frac{8\beta^2}{c_p\mu}+ \frac{4\sigma_A^2}{c_d\mu}},\frac{1-2\beta^2/(\mu c_p)}{\frac{4\beta^2}{c_p^2}+4+\frac{2\sigma_A^2}{c_d c_p}},\frac{1-1/(\rho c_d)}{\frac{4\rho \sigma_A^2}{c_p}} \right)
\]
Because $c_d>1/\rho$ and $c_p>\frac{2\beta^2}{\mu}$, we have
\begin{align*}
    & \frac{8\beta^2}{c_p\mu} \leq 4,~~\frac{4\sigma_A^2}{c_d\mu}\leq \frac{4\rho \sigma_A^2}{\mu},\\
    & 4\beta^2/c_p^2 \le 4\beta^2 \cdot\frac{\mu^2}{4\beta^4}\le 1,\quad \frac{2\sigma_A^2}{c_dc_p}\le \frac{2\rho \mu \sigma_A^2}{2\beta^2}\le \frac{\rho \sigma_A^2}{\mu},\\
    &\frac{4\rho \sigma_A^2}{c_p} \leq 4\rho \sigma_A^2 \cdot \frac{\mu}{2\beta^2} \leq \frac{2\rho\sigma_A^2}{\mu},
\end{align*}
implying that the denominators in all three terms inside the $\min$ bracket are smaller than or equal to $5(1+\rho\sigma_A^2/\mu)$. Therefore,
\begin{align*}
    \alpha' \ge \frac{1}{5c_pc_d(1/\sigma_A^2+\rho/\mu)}\min \left(1-2\beta^2/(\mu c_p),1-1/(\rho c_d)\right) = \alpha,
\end{align*}
which, together with \eqref{eq:three_term_sum} and \eqref{summable 2 BMM-PD}, results in \eqref{eq:BMM_result_linear}.

\subsection{Proof of Theorem \ref{thm:BMM-D}}\label{appen:thm_BMM-D}

By Assumption \ref{asm:dual_modelMM}(b), we have
\begin{equation}\label{Lk1s}
    \begin{split}
     L^k_\rho(v^{k+1}) =& -\hat{q}_\rho^k(v^{k+1}) + \frac{c_d}{2}\|v^{k+1}-v^k\|^2 \\
    \geq& -q_\rho(v^k) + \langle -\nabla q_\rho(v^k),v^{k+1}-v^k\rangle +\frac{c_d}{2}\|v^{k+1}-v^k\|^2,
    \end{split}
\end{equation}
where $L_\rho^k$ is defined in Appendix \ref{appen:thm_BMM}. Since $F$ is proper, closed, and convex, the negative dual function $-q_\rho(v)=- \inf_x \{ F(x) + \langle v,Ax-b\rangle +\frac{\rho}{2}\|Ax-b\|^2 \}$ is convex and $\frac{1}{\rho}$-smooth \cite[Theorem 1]{li2025smoothnessaugmentedlagrangiandual}, which gives
\[
-q_\rho(v^{k+1}) \leq -q_\rho(v^k) + \langle -\nabla q_\rho(v^k),v^{k+1}-v^k \rangle + \frac{1}{2\rho}\|v^{k+1}-v^k\|^2.
\]
Substituting above equation into \eqref{Lk1s} and using \eqref{eq:MM_x}, we have
\begin{equation}\label{Lk1}
    \begin{split}
     L^k_\rho(v^{k+1}) \geq& -q_\rho(v^k) + \langle -\nabla q_\rho(v^k),v^{k+1}-v^k\rangle +\frac{c_d}{2}\|v^{k+1}-v^k\|^2 \\
    \geq& -q_\rho(v^{k+1}) + (\frac{c_d}{2}-\frac{1}{2\rho})\|v^{k+1}-v^k\|^2  \\
    =& -F(x^{k+2}) - \langle v^{k+1},Ax^{k+2}-b \rangle - \frac{\rho}{2}\|Ax^{k+2}-b\|^2 \\
    &+ (\frac{c_d}{2}-\frac{1}{2\rho})\|v^{k+1}-v^k\|^2.
    \end{split}
\end{equation}
By Assumption \ref{asm:dual_modelMM} (c) and the strong duality between problem \eqref{problem} and its dual problem, we have 
\begin{equation}\label{Lk2}
\begin{split}
    L^k_\rho(v^\star) &= -\hat{q}_\rho^k(v^\star) + \frac{c_d}{2}\|v^\star-v^k\|^2 \\
    &\leq -q_\rho(v^\star) + \frac{c_d}{2}\|v^\star-v^k\|^2 \\
    &= -F(x^\star) + \frac{c_d}{2}\|v^\star-v^k\|^2.
\end{split} 
\end{equation}
Substituting \eqref{Lk1} and \eqref{Lk2} into \eqref{eq:Lk1-Lk2_BMM} gives
\begin{equation}\label{A}
    \begin{split}
        &\quad F(x^\star) -F(x^{k+2}) - \langle v^{k+1},Ax^{k+2}-b \rangle - \frac{\rho}{2}\|Ax^{k+2}-b\|^2  \\
    &\leq \frac{c_d}{2}\|v^\star-v^k\|^2 - \frac{c_d}{2}\|v^\star-v^{k+1}\|^2-(\frac{c_d}{2}-\frac{1}{2\rho})\|v^{k+1}-v^k\|^2. 
    \end{split}
\end{equation}
By the first-order optimality condition of the primal update \eqref{eq:MM_x}, we have that for some $g_F^{k+2}\in \partial F(x^{k+2})$, 
\begin{equation}\label{eq:BMM-D-firsto}
    g_F^{k+2} + \rho A^T(Ax^{k+2}-b) + A^Tv^{k+1} = \mathbf{0},
\end{equation}
which together with $b=Ax^\star$ implies 
\begin{equation}\nonumber
    \begin{split}
        \langle v^{k+1},Ax^{k+2}-b \rangle =& \langle A^Tv^{k+1},x^{k+2}-x^\star \rangle \\
        =& \langle -g_F^{k+2}-\rho A^T(Ax^{k+2}-b),x^{k+2}-x^\star \rangle \\
        =&\langle g_F^{k+2},x^\star-x^{k+2} \rangle - \rho \|Ax^{k+2}-b\|^2.
    \end{split}
\end{equation}
Substituting above equation into \eqref{A} and using $c_d \geq \frac{1}{\rho}$, we have 
\begin{equation}\label{eq:BMM-D_summable_Ax-b}
\begin{split}
    &F(x^\star) - F(x^{k+2}) - \langle g_F^{k+2},x^\star-x^{k+2}\rangle + \frac{\rho}{2}\|Ax^{k+2}-b\|^2 \\
    \leq& \frac{c_d}{2}\|v^\star-v^k\|^2 - \frac{c_d}{2}\|v^\star-v^{k+1}\|^2.
\end{split}
\end{equation}
Using telescoping cancellation on the above equation yields that for all $K\geq 1$,
\begin{equation}\nonumber
   \frac{\rho}{2} \sum_{k=1}^K \|Ax^{k+1}-b\|^2 \leq \frac{c_d}{2}\|v^\star-v^0\|^2,
\end{equation}
which implies $\lim_{k \to +\infty}\|Ax^{k+1}-b\|=0$ and \eqref{eq:BMM_result_sublinear_Ax-b} according to \cite[Proposition 3.4]{shi2015extra}. 
From \eqref{eq:BMM-D-firsto}, we can know that
\begin{equation}\label{eq:BMM-D-g_in_rangeA}
g_F^{k+2} = -\rho A^T(Ax^{k+2}-b) - A^Tv^{k+1} \in \operatorname{Range}(A^T).  
\end{equation}
Since $P_N$ is the projection matrix onto $\operatorname{Null}(A)$ and $\operatorname{Null}(A)\perp \operatorname{Range}(A^T)$, by \eqref{eq:BMM-D-g_in_rangeA} we have
\[
P_Ng_F^{k+2} = \mathbf{0}.
\]

Next, we derive the linear convergence. 
By $h(x) \equiv0$ and the smoothness of $f$, the left-hand side of \eqref{eq:BMM-D_summable_Ax-b} satisfies, 
\begin{equation}\nonumber
    \begin{split}
        &F(x^\star) - F(x^{k+2}) - \langle g_F^{k+2},x^\star-x^{k+2}\rangle + \frac{\rho}{2}\|Ax^{k+2}-b\|^2 \\
        =& f(x^\star) - f(x^{k+2}) -  \langle \nabla f(x^{k+2}),x^\star-x^{k+2}\rangle + \frac{\rho}{2}\|Ax^{k+2}-b\|^2  \\
        \geq & \frac{1}{2\beta}\|\nabla f(x^{k+2})-\nabla f(x^\star)\|^2 + \frac{\rho}{2}\|Ax^{k+2}-b\|^2,
    \end{split}
\end{equation}
where the last step uses \cite[(2.1.10)]{nesterov2018lectures}.
Substituting above equation into \eqref{eq:BMM-D_summable_Ax-b}, gives
\begin{equation}\label{1/k}
    \begin{split}
        &\frac{1}{2\beta}\|\nabla f(x^{k+2})-\nabla f(x^\star)\|^2 + \frac{\rho}{2}\|Ax^{k+2}-b\|^2 \\
        \leq& \frac{c_d}{2}\|v^\star-v^k\|^2 - \frac{c_d}{2}\|v^\star-v^{k+1}\|^2.
    \end{split}
\end{equation}
By the first-order optimality condition of the primal update \eqref{eq:MM_x} and the KKT condition of problem \eqref{problem}, we can get
\begin{equation}\label{eq:BMM-D_opKKT}
    \nabla f(x^{k+2}) +\rho A^T(Ax^{k+2}-b) + A^Tv^{k+1} = \mathbf{0},~~\nabla f(x^\star) + A^Tv^\star = \mathbf{0}.
\end{equation}
Letting $(A^T)^\dag$ be the Moore-Penrose pseudo-inverse of $A^T$, according to \cite[Proposition 29, Page 238]{petersen2012linear}, $(A^T)^\dag A^T$ is the projection matrix onto $\operatorname{Range}(A)$, which, together with \eqref{eq:BMM_vk_vstar_range} and \eqref{eq:BMM-D_opKKT}, implies
\begin{equation*}
    \begin{split}
        v^{k+1}-v^\star =& (A^T)^\dagger A^T(v^{k+1}-v^\star) \\
        =& (A^T)^\dagger (\nabla f(x^\star) - \nabla f(x^{k+2})) - \rho (A^T)^\dagger A^T(Ax^{k+2}-b) \\
        =&(A^T)^\dagger(\nabla f(x^\star) - \nabla f(x^{k+2})) - \rho (Ax^{k+2}-b).
    \end{split}
\end{equation*}
By the above equation, the AM-GM inequality, and $\|(A^T)^\dagger\| = \frac{1}{\sigma_A}$, we have
\begin{equation}\label{eq:BMM-D_ATv}
    \begin{split}
        \frac{c_d}{2}\|v^{k+1}-v^\star\|^2 =& \frac{c_d}{2}\|(A^T)^\dagger(\nabla f(x^\star) - \nabla f(x^{k+2})) - \rho (Ax^{k+2}-b)\|^2 \\
        \leq& c_d\|(A^T)^\dagger(\nabla f(x^\star) - \nabla f(x^{k+2}))\|^2 + c_d\rho^2\|Ax^{k+2}-b\|^2 \\
        \leq& \frac{c_d}{\sigma_A^2}\|\nabla f(x^\star) - \nabla f(x^{k+2})\|^2 + c_d\rho^2\|Ax^{k+2}-b\|^2.
    \end{split}
\end{equation}
By $\alpha = \min (\frac{ \sigma_A^2}{2c_d\beta},\frac{1}{2c_d\rho})$ and \eqref{eq:BMM-D_ATv},
\begin{align*}
    &\quad \frac{1}{2\beta}\|\nabla f(x^{k+2})-\nabla f(x^\star)\|^2 +\frac{\rho}{2}\|Ax^{k+2}-b\|^2\\
    &\geq \alpha \cdot \left( \frac{c_d}{\sigma_A^2}\|\nabla f(x^\star) - \nabla f(x^{k+2})\|^2 + c_d\rho^2\|Ax^{k+2}-b\|^2 \right)\\
    &\geq \alpha \cdot \frac{c_d}{2}\|v^{k+1}-v^\star\|^2.
\end{align*}
Substituting above inequality into \eqref{1/k}, gives
\[
\|v^{k+1}-v^\star\|^2 \leq \frac{1}{(1+\alpha)}\|v^k-v^\star\|^2 \leq \frac{1}{(1+\alpha)^{k+1}}\|v^0-v^\star\|^2,
\]
i.e., \eqref{eq:BMMD_v_rate} holds. By \eqref{eq:BMMD_v_rate} and \eqref{1/k}, we have \eqref{eq:BMMD_gra_rate}--\eqref{eq:BMMD_feas_rate}.

\end{appendices}

\bibliography{sn-bibliography}

\end{document}